\documentclass[final]{siamltex}
\usepackage{algorithm}

\newtheorem{remark}{remark}
\usepackage{epsfig}

\title{Space-time domain decomposition method for scalar conservation laws 
\thanks{Institute of mathematics University of Neuch\^atel CH-2009 (Switzerland)}} 

\author{Souleymane Doucour\'e \thanks{(souleymane.doucoure@unine. ch)}}

\begin{document}
\maketitle

\begin{abstract}
The Space-Time Integrated Least-Squares (STILS) method is considered to analyze a space-time domain decomposition algorithm for scalar conservation laws. Continuous and discrete convergence estimates are given. 
 Next using a time-marching finite element formulation,  
 the STILS solution and its domain decomposition form are numerically compared. 
\end{abstract}


\begin{AMS}\end{AMS}

\pagestyle{myheadings}
\thispagestyle{plain}
\markboth{SOULEYMANE DOUCOURE}{SPACE-TIME DDM}

\section{Introduction}
The need of using domain decomposition methods (DDM) when solving partial differential equations is nowadays more and more obvious. Essentially, it consists in splitting the domain into several parts and in solving a family of problems within the subdomains: this family is equivalent to the original problem if and only if suitable conditions holds along the interface separating the subdomains. This technique is motivated essentially by numerical approximation (for instance, computational complexity can be reduced and geometrical features can be exploited thoroughly). However, the study of convergence of any discrete scheme usually needs a careful investigation of the properties of the continuous solution. The classical way of applying domain decomposition methods to evolution problems is to discretize the time dimension first uniformly over the whole domain by an implicit scheme, and to apply domain decomposition at each time step separately to solve the sequence of steady problems obtained from the implicit time discretization. Theoretical and numerical analysis of this approach with non-overlapping subdomains for hyperbolic problems as transport equation, scalar conservation laws can be found in \cite{quv,gast,gast1}. Another way to consider the domain decomposition approach is given by the overlapping Schwarz method which is applied notably to the heat equation see \cite{Meu91}, the convection-diffusion equation see \cite{Cai91} and \cite{Cai94}. The possibility to use different grids in space for each subdomain is analyzed in \cite{BGT97}. But due to the uniform time discretization, one cannot have an optimal space-time discretization. Thus this approach has a significant disadvantage to use a uniform time discretization over the entire domain and thus loses one of the main features of domain decomposition  methods, namely to treat the problem on each subdomain numerically differently, with an appropriate discretization both in time and space adapted to the subdomain problems. For these reasons space-time domain decomposition methods are mandatory, like 
Schwarz Waveform Relaxation (SWR) method, which is widely used. The scheme also uses an overlapping domain decomposition in space, like the classical Schwarz algorithm for steady problems see \cite{sch70}, but then the algorithm solves evolution problems on the subdomains and uses an iteration to converge to the solution of the original problem.  Convergence properties of the SWR method 
applied to convection dominated viscous conservation laws with nonlinear flux are given in \cite{Gand}.\\
This work is a contribution to the non-overlapping case in dealing with a new space-time DDM for scalar conservation laws without viscous terms and using the STILS method. 
The STILS approach for PDEs allows to stabilize the numerical finite element solutions.  
 \\
More generally, the method of least-squares is a standard approach to the approximate solutions of PDEs see for example in \cite{glow,boch}. This method allows to convexify noncoercive problems. The basic idea is the same in finite dimensional case: a nondegenerate system is transformed in a positive definite system. 
In \cite{hugfr}, it has allowed to perform the Stream Upwind Petrov Galerkin (SUPG) method \cite{brooh} in the Galerkin Least-Squares (GLS) method. In the same spirit,  the Streamline Diffusion and Discontinuous Galerkin \cite{john,joh} have mutated to give Characteristic Streamline Diffusion (CSD) \cite{hans}. The space-time framework is presented for example in \cite{chat,nguy}. The common principle of these methods is to stabilize the problem in adding a quadratic residual to the advective terms. 
Next the STILS method (without to add a quadratic term) is introduced in \cite{ap,azr,bessp} for the transport equation with more convincing results. An existence-uniqueness analysis of the STILS solution and his comparison with renormalized solutions have been detailed in \cite{bessop} for linear conservation laws when the velocity has a low regularity. Finally in the reference \cite{besg} the time-marching version of the STILS method has given well numerical results compared to the solutions obtained in characteristic method or in the usual  finite element methods (Galerkin, Streamline Diffusion, Shock Capturing). 
\\ 
In section $3,$ the original problem of scalar conservation laws and its two-domain form are formulated in a space-time form.  These STILS formulations with an equivalence result and an iteration-by-subdomain scheme are presented in section $4$. In section $5,$ the established results are generalized to the multi-domains case in time. 
In section $6,$ discrete convergence estimates are given with finite element methods in time and space. Finally in section $7,$ the STILS and STILS-DDM solutions are numerically compared for the Hansbo example providing in \cite{hans}. 
\section{Problem statement}
Let $\Omega\subset\mathbf{R}^d$ be a domain with a Lipschitz boundary $\partial\Omega$ satisfying the cone property, and let $T > 0$. Consider an advection velocity $u: \Omega\times]0,T[\longrightarrow\mathbf{R}^d,$ and $f\in L^2(\Omega\times]0,T[)$ a given source term. In the sequel we denote by $(.\vert.)$ the euclidian scalar product. \\
The problem consists in finding a function $c:\Omega\times]0,T[\rightarrow\mathbf{R}$ satisfying the following partial differential equation
\begin{eqnarray}
\frac{\partial c}{\partial t}  + \mathrm{div}(cu)& = & f  
\label{st1}
\end{eqnarray}
and the initial and boundary conditions 
\begin{eqnarray}
c(x,0)&=&c_0(x)\;\;\mathrm{for}\;x\;\mathrm{in}\;\;\Omega
\label{st2}\\
c(x,t)&=&c_{00}(x,t)\;\;\mathrm{for}\;x\;\mathrm{in}\;\;\partial\Omega^-\times(0,T)
\label{st3}
\end{eqnarray}
where
$\partial\Omega^-=\{ x\in\partial\Omega: (u(x,t)\vert n(x))<0 \}$, $n(x)$ is the outer normal to $\partial\Omega$ at point $x$. For the sake of the presentation, it is assumed that $\partial\Omega^-$ is not dependent on t.\\
The function $c$ represents for example a chemical concentration transported by a liquid flow in the domain $\Omega$. The velocity $u$ may be a solution of Navier-Stokes equations. It may too derive from a potential. For example in a porous medium context, $u$ will be a Darcy velocity.\\
Let us now partition $\Omega$ into two non-overlapping subdomains $\Omega_1$ and $\Omega_2$ with Lipschitz continuous boundaries $\partial\Omega_1$ and $\partial\Omega_2$. The common interface $\partial\Omega_1\cap\partial\Omega_2$ is denoted by $\Gamma$; the normal unit vector on $\Gamma$ pointing into $\Omega_2$ is denoted by $n$.
\begin{eqnarray}
\overline\Omega=\overline\Omega_1\cup\overline\Omega_2,\quad\Gamma=\partial\Omega_1\cap\partial\Omega_2,\quad\partial\Omega_i^-=\partial\Omega^-\cap\partial\Omega_i,\;i=1,2.
\label{st4}
\end{eqnarray}
Then the original problem is rewritten in a two-domains form. More precisely for $i=1,2$ we consider the problem
\begin{eqnarray}
\frac{\partial c^i}{\partial t}  + \mathrm{div}(c^iu)& = & f
\label{st5}
\end{eqnarray}
with
\begin{eqnarray}
c^i(x,0) & =&  c_{0}(x) \quad\mathrm{for}\;\;x\;\; \mathrm{in}\; \Omega_i
\label{st6}\\
c^i(x,t) & = & c_{00}(x,t) \quad\mathrm{for}\;\;x\; \mathrm{on}\; \; \partial\Omega^-_i\times(0,T)
\label{st7}
\end{eqnarray}
and the following interface condition
\begin{eqnarray}
c^1&=&c^2\quad\mathrm{for}\;\;x\;\;\mathrm{on}\;\Gamma\times(0,T).
\label{st8}
\end{eqnarray}

The equivalence between the two above problems is proved in the steady case (see \cite{gast,gast1}).\\
The originality of the STILS method is to solve the equation (\ref{st1}) in both directions (space and time) together and in $L^2$ sense. 
In the next sections the principles of the STILS method \cite{ap,azr,bessp,bessop,besg} are recalled for the equations (\ref{st1})-(\ref{st3})
, and its two-domains form (\ref{st5})-(\ref{st8}) is also equivalently formulated.
\section{Space-time framework}
In a space-time description, time is just seen as the (d+1)-th space dimension. Introduce the following space-time objects:
the domain $Q= \Omega\times]0,T[$, with
the boundary $\left(\partial\Omega\times]0,T[\right)\cup\left(\Omega\times\{T\}\right)\cup\left(\Omega\times\{0\}\right)$, and
the outer-normal $\widetilde n$ is given by $$\widetilde n=\left\{
  \begin{tabular}{ll}
  $(n,0)  \;\; \; \;\;\;\; \mathrm{on} \quad \partial\Omega\times]0,T[ $\\
  $(0,-1) \quad\; \mathrm{on}\quad \Omega\times\{0\} $\\
  $(0,1)\quad\;\;\;\; \mathrm{on}\quad  \Omega\times\{T\}$
  \end{tabular}
  \right.$$
 and the space-time velocity is $\widetilde u=(u_1,...,u_d,1)^t$. The space-time inflow boundary is
 $$\partial Q^-=\{(x,t)\in\partial Q , (\widetilde u\vert \widetilde n)<0\}=\left(\partial\Omega^-\times]0,T[\right)\cup\left(\Omega\times\{0\}\right).$$
   In this context, initial condition and inflow condition become 
   $$ c_b(x,t)=\left\{
  \begin{tabular}{ll}
  $c_0(x) \quad if \quad (x,t)\in\Omega\times\{0\} $\\
  $c_{00}(x,t) \quad if \quad (x,t)\in\partial\Omega^-\times(0,T).$
  \end{tabular}
  \right.$$
  For a sufficiently regular function $\varphi$ defined on $Q$, $\nabla_t\varphi=\left(\frac{\partial\varphi}{\partial x_1},\frac{\partial\varphi}{\partial x_2},\cdots,\frac{\partial\varphi}{\partial x_d},\frac{\partial\varphi}{\partial t}\right)$ is the space-time gradient, and $\nabla_t\cdot(\varphi\widetilde u)=\frac{\partial\varphi}{\partial t}+\displaystyle\sum_{i=1}^d\frac{\partial}{\partial x_i}(u_i\varphi).$
   With the above notations, the problem (\ref{st1})-(\ref{st3}) transforms in space-time form
   \begin{eqnarray} 
  \nabla_t\cdot( c\widetilde u)&=& f
\label{st9}\\
c& = & c_b \quad \mathrm{for}\;(x,t)\;\;\mathrm{on}\; \partial Q^-.
\label{st10}
\end{eqnarray}
And the space-time counterpart of the two-domains problem (\ref{st5})-(\ref{st8}) is given by
  \begin{eqnarray} 
 \nabla_t\cdot( c^1\widetilde u)&=& {f}
\label{st11}\\
c^1& = & c_{b,1} \quad \mathrm{for}\;(x,t)\;\;\mathrm{on}\;\; \partial Q^-_1
\label{st12}\\
c^1&=&c^2  \;\quad\; \mathrm{for}\;(x,t)\;\;\mathrm{on}\;\; \widetilde\Gamma 
\label{st13}\\
 \nabla_t\cdot( c^2\widetilde u)&=& f
\label{st14}\\
c^2& = & c_{b,2} \quad \mathrm{for}\;(x,t)\;\;\mathrm{on}\;\; \partial Q^-_2
\label{st15}
\end{eqnarray}
where (see fig. \ref{fig1})
$$Q_i=\Omega_i\times(0,T),\; 
\overline Q=\overline Q_1\cup\overline Q_2,\;
\partial Q_i^-=\partial Q^-\cap\partial Q_i,$$
\begin{eqnarray}
\nonumber
\widetilde\Gamma\;\;&=&\partial Q_1\cap\partial Q_2=\widetilde\Gamma^-\cup\widetilde\Gamma^+\cup\widetilde\Gamma^0\\\nonumber
\widetilde\Gamma^-&=&
\{ (x,t)\in\widetilde\Gamma: (\widetilde u(x,t)\vert\widetilde n(x,t))<0 \}=\widetilde\Gamma^-_1=\widetilde\Gamma^+_2
\nonumber
\\
\widetilde\Gamma^+&=&
\{ (x,t)\in\widetilde\Gamma: (\widetilde u(x,t)\vert\widetilde n(x,t))>0 \}=\widetilde\Gamma^-_2=\widetilde\Gamma^+_1
\nonumber
\\
\widetilde\Gamma^0\;&=&\{(x,t)\in\widetilde\Gamma: (\widetilde u(x,t)\vert\widetilde n(x,t))=0 \}\nonumber
\end{eqnarray}
  \begin{figure}[H]
\begin{center}
\includegraphics[width=6.cm]{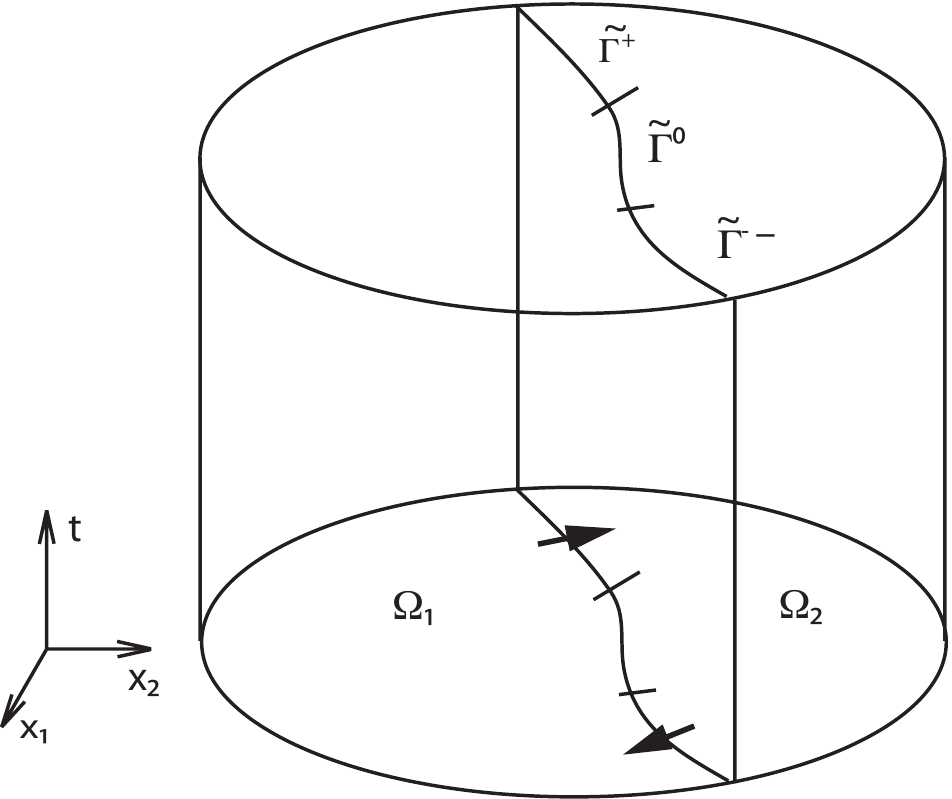}
\caption{Decomposition of $Q$ and explanation of the symbology}
\label{fig1}
\end{center}
\end{figure}


In the sequel the velocity has the following regularity
\begin{eqnarray}
u\in\;L^\infty(Q)^d,\;\nabla\cdot u\in\;L^\infty(Q).
\label{st16}
\end{eqnarray}
The following theorem is proved in \cite{chen}.
\begin{theorem}
\label{chen_frid}
Under the assumption (\ref{st16}) the normal trace of $u,$ $(\widetilde u\vert \widetilde n)$ is in $L^\infty(\partial Q)$.
\label{th1}
\end{theorem}

\section{STILS method}
\subsection{Global formulation}
For $\varphi\in\mathcal{D}(\overline Q)$, consider the norm $$\Vert \varphi\Vert_{H(u,Q)}=\left(\Vert\varphi\Vert_{L^2(Q)}^2+\Vert \nabla_t\cdot(\widetilde u\varphi) \Vert^2_{L^2(Q)}+\int_{\partial Q^-}\vert(\widetilde u\vert\widetilde n)\vert\varphi^2\mathrm{d}\tilde\sigma \right)^{\frac{1}{2}} $$ see \cite{bessop}. Define the anisotropic Sobolev space $H(u,Q)$ as the closure of $\mathcal{D}(\overline Q)$ for this norm $$H(u,Q)=\overline{\mathcal{D}(\overline Q)}^{H(u,Q)}.$$ 
Under the assumption (\ref{st16}), there exists a linear continuous trace operator \cite{bessop}
\begin{eqnarray}
\nonumber
 \gamma^-:H(u,Q)&\rightarrow& L^2(\partial Q^-,\vert(\widetilde u\vert\widetilde n)\vert\mathrm{d}\widetilde\sigma)\\\nonumber\varphi&\mapsto&\gamma^-(\varphi)={\varphi_\mid}_{\partial Q^-}.
 \nonumber
\end{eqnarray}
Finally define the spaces
\begin{eqnarray}
\nonumber
H_0(u,Q,\partial Q^-)&=&\{\varphi\in H(u,Q),\, {\varphi_\mid}_{\partial Q^-}=0\}\\
\nonumber
V^-&=&\{\varphi\in H(u,Q),\;\gamma^-(\varphi)\in L^2(\partial Q^-,\vert(\widetilde u\vert\widetilde n)\vert\mathrm{d}\widetilde\sigma)\}\\
\nonumber
G^-&=&\gamma^-(V^-).
\end{eqnarray}
The following theorem is proved in \cite{bessp,bessop}.
\begin{theorem}(Curved-Poincar\'e inequality)\label{poincare}\\
If $u\in L^\infty(Q)^d$ and $\nabla\cdot u\in L^\infty(Q)$, the semi-norm on $H(u,Q)$ defined by $$\vert\varphi\vert_{H(u,Q)}= \Big(\int_Q(\nabla_t\cdot(\widetilde u\varphi))^2\;\mathrm{dxdt} +\int_{\partial Q^-}\vert(\widetilde u\vert\widetilde n)\vert\varphi^2\mathrm{d}\widetilde\sigma\Big)^{\frac{1}{2}}$$
is a norm equivalent to the norm on $H(u,Q).$
\label{th2}
\end{theorem}
\\

A solution of equation (\ref{st1}) in $L^2$ corresponds to a minimizer in $\{\varphi\in H(u,Q);\;\gamma^-(\varphi)-c_b=0\}$ of the following convex, H(u,Q)-coercive functional
\begin{eqnarray}
J(c)=\frac{1}{2}\left(\int_Q \left(\nabla_t\cdot(\widetilde uc)-f\right)^2\;\mathrm{dxdt}-\int_{\partial Q^-}c^2(\widetilde u\vert\widetilde n)\mathrm{d}\tilde\sigma\right).
\label{st17}
\end{eqnarray}
Its Gateaux derivate is 
\begin{eqnarray}
 [DJ(c)](\varphi)=\int_Q\left(\nabla_t\cdot(\widetilde uc)-f\right)\nabla_t\cdot(\widetilde u\varphi)\;\mathrm{dxdt} -\int_{\partial Q^-}c\varphi(\widetilde u\vert\widetilde n)\mathrm{d}\widetilde\sigma.
\label{st18}
\end{eqnarray}
Therefore a sufficient condition to get the least-squares solution of (\ref{st1})-(\ref{st3}) is the following weak formulation: if $c_b\in G^-,$ find $c\in H(u,Q),\;\gamma^-(c)=c_b$ such that
\begin{eqnarray}
\begin{array}{lcr}
\int_Q \nabla_t\cdot(\widetilde uc)\nabla_t\cdot(\widetilde u\varphi)\;\mathrm{dxdt} =\int_Q f\;\nabla_t\cdot(\widetilde u\varphi)\;\mathrm{dxdt}
\quad\forall \varphi\in H_0(u,Q,\partial Q^-).
\label{st19}
\end{array}
\end{eqnarray}

The maximum principle and weak-solution results proved in \cite{bessop} are given by
\begin{theorem}(Maximum principle)\label{princ_max}
Assume that $\Omega$ is bounded, and the function $f=0$ in (\ref{st1}). Let  $c_b\in G^-\cap L^\infty(\partial Q^-)$, if $\nabla\cdot u=0$ the solution c of  (\ref{st19}) satisfies:
\begin{eqnarray}
\inf c_b \leq c \leq \sup c_b.\nonumber
\end{eqnarray}
\label{th3}
\end{theorem}
\begin{theorem}(Weak-solution)\label{sol_fa}
For $f\in L^\infty(Q)$ and $u\in L^1(0,T;BV(\Omega)^d)$ with $\nabla\cdot u=0,$ let $c\in H_0(u,Q,\partial Q^-)$ be the solution of $$\int_Q \nabla_t\cdot(\tilde uc)\nabla_t\cdot(\widetilde u\varphi)\;\mathrm{dxdt} =\int_Q f\nabla_t\cdot(\widetilde u\varphi)\;\mathrm{dxdt}\;\forall \varphi\in H_0(u,Q,\partial Q^-)$$
then it is a weak solution, i.e: 
\begin{eqnarray}
\int_Q \nabla_t\cdot(\widetilde uc)\varphi\;\mathrm{dxdt} =\int_Q f\varphi\;\mathrm{dxdt}\quad \forall \varphi\in L^2(Q).
\label{th4}\nonumber
\end{eqnarray}
\end{theorem}

For $c_b\in G^-$, let $C_b\in H(u,Q)$ be such that $\gamma^-(C_b)=c_b.$ Then $\rho=c-C_b\in H_0(u,Q,\partial Q^-)$ is the unique solution (Curved-Poincar\'e  inequality (theorem \ref{poincare}) and Lax-Milgram lemma) of
\begin{eqnarray}
\hspace{9mm}\int_Q \nabla_t\cdot(\widetilde u\rho)\nabla_t\cdot(\widetilde u\varphi)\;\mathrm{dxdt} =\int_Q f\nabla_t\cdot(\widetilde u\varphi)\;\mathrm{dxdt}-\int_Q \nabla_t\cdot(\widetilde uC_b)\nabla_t\cdot(\widetilde u\varphi)\;\mathrm{dxdt}
\label{st20}
\end{eqnarray}
$\forall \varphi\in H_0(u,Q,\partial Q^-)$. 
Set now 
$$W(u,Q)=\{\varphi\in L^2(Q),\;\nabla_t\cdot(\widetilde u\varphi)\in L^2(Q),\;{\varphi_\mid}_{\partial Q^-}\in L^2(\partial Q^-,\vert(\widetilde u\vert\widetilde n)\vert\mathrm{d}\widetilde\sigma)\}.$$
If $u$ is regular enough, e.g. $u\in L^2(0,T;H^1_0(\Omega)^d)$, it can be seen that $$H(U,Q)\cap L^\infty(Q)=W(U,Q)\cap L^\infty(Q)$$
see \cite{bessop}.
\\
In the sequel we assume 
 $c_b\in L^\infty(\partial Q^-)\cap G^-$ such that $c_b=\gamma^-(C_b)$ with $C_b\in W^{1,\infty}(Q).$ Thus, if $f=0$ in (\ref{st1}) the least-squares solution $c\in L^\infty(Q).$
\subsection{Two-domains formulation}

In this subsection the two-domains problem (\ref{st11})-(\ref{st15}) is rewritten in STILS form.
\subsubsection{Anisotropic Sobolev spaces}
Consider the norm
$$\Vert \varphi_i\Vert_{H(u,Q_i)}=\left(\Vert\varphi_i\Vert_{L^2(Q_i)}^2+\Vert \nabla_t\cdot(\widetilde u\varphi_i) \Vert^2_{L^2(Q_i)}+\int_{\partial Q^-_i\cup\widetilde\Gamma_i^-}\vert (\widetilde u\vert\widetilde n) \vert{\varphi_i}^2d\widetilde\sigma\right)^{\frac{1}{2}},$$
and define $H(u,Q_i)$ as the closure of $\mathcal{D}(\overline Q_i)$ for this norm
$$H(u,Q_i)=\overline{\mathcal{D}(\overline Q_i)}^{H(u,Q_i)}
,\;H_0(u,Q_i,\widetilde\Sigma)=\{\varphi_i\in H(u,Q_i);\; {\varphi_i}_{\mid_{\widetilde\Sigma}}=0\},\;\;\widetilde\Sigma\subset\partial Q_i^-\cup\widetilde\Gamma^-_i .
$$
Moreover if $\widetilde\Sigma$ is a (relatively) open set of $\partial Q_i^-\cup\widetilde\Gamma^-_i$ or $\partial Q_i^+\cup\widetilde\Gamma^+_i$, set  $$L^2(\widetilde\Sigma,\vert(\widetilde u\vert\widetilde n)\vert\mathrm{d}\widetilde\sigma)=\{\varphi:\widetilde\Sigma\rightarrow\mathbf{R}; \int_{\widetilde\Sigma}\vert(\widetilde u\vert\widetilde n)\vert\varphi^2\mathrm{d}\widetilde\sigma<\infty\},$$
this is a Hilbert space with the norm
$$\Vert\varphi \Vert_{\widetilde\Sigma,(\widetilde u\vert\widetilde n)}
=\Big(\int_{\widetilde\Sigma}\vert(\widetilde u\vert\widetilde n)\vert\varphi^2\mathrm{d}\widetilde\sigma\Big)^{1/2}$$
induced by the scalar product$$(\varphi_i,\psi_i)_{\widetilde\Sigma,(\widetilde u\vert\widetilde n)}=\pm\int_{\widetilde\Sigma}(\widetilde u\vert\widetilde n)\varphi_i\psi_i\mathrm{d}\widetilde\sigma,\;\;\widetilde\Sigma\subset \partial Q_i^\pm\cup\widetilde\Gamma^\pm_i. $$

For $i=1,2,$ consider the Hilbert space $$H(\widetilde\mathrm{div},Q_i)=\{\varphi_i\in L^2(Q_i)^{d+1}:\nabla_t\cdot\varphi_i\in L^2(Q_i)\}$$
with the graph norm $$\Vert \varphi_i\Vert_{H(\widetilde\mathrm{div},Q_i)}=\Big[\Vert \varphi_i\Vert^2_{{L^2(Q_i)}^{d+1}}+\Vert\nabla_t\cdot\varphi_i\Vert^2_{L^2(Q_i)}\Big]^{1/2}.$$
The same definition is given to $H(\widetilde\mathrm{div},Q).$\\

The trace result for the $H(u,Q_i)-$functions is given by
\begin{proposition}\label{trace_loc}
Under the assumption (\ref{st16}) there exists a linear continuous trace operator 
 $$\begin{array}{ccc}
            \gamma_i:H(u,Q_i) & \longrightarrow & L^2(\partial Q_i,\vert(\widetilde u\vert\widetilde n)\vert d\widetilde\sigma) \\
                     \varphi_i  & \longmapsto     & {\varphi_i}_{\mid_{\partial Q_i}}
           \end{array}$$
which can be localized as
$$\begin{array}{ccc}
            {\gamma_i}^{\pm}:H(u,Q_i) & \longrightarrow & L^2(\partial Q_i^\pm\cup\widetilde\Gamma^\pm_i,\vert(\widetilde u\vert\widetilde n)\vert d\widetilde\sigma) \\
                     \varphi_i  & \longmapsto     & {\varphi_i}_{\mid_{\partial Q_i^\pm\cup\widetilde\Gamma^\pm_i}}
           \end{array}$$
 \label{th5}          
\end{proposition}
\begin{proof}
The detailed proof can be found in \cite{bessop}: 
first, consider the well-know trace operator 
$\gamma_n$ defined from $H(\mathrm{\widetilde{div}},Q_i)$ 
with value in $H^{-\frac{1}{2}}(\partial Q_i)$ (see \cite{gr,brez})
$$v_i\mapsto (\widetilde n\vert v_i)_{\mid
_{\partial Q_i}},\;\mathrm{for}\;\mathrm{all}\;v_i\in H(\widetilde{\mathrm{div}},Q_i)$$
with associated Green formula
$$\int_{Q_i}\left( \nabla_t\cdot(v_i)\psi_i + ( v_i\vert\nabla_t\psi_i)\right)\mathrm{dxdt}=\Big\langle(v_i\vert\widetilde n),\psi_i\Big\rangle_{H^{\frac{-1}{2}}(\partial Q_i);H^{\frac{1}{2}}(\partial Q_i)},\;
\forall\psi_i\in H^1(Q_i).$$ Plugging $v_i=\widetilde u\rho_i$ in the previous formula, we have
$$\int_{Q_i}\left(\nabla_t\cdot(\widetilde u\rho_i)\psi_i + ( \widetilde u\vert\nabla_t\psi_i)\rho_i\right)\;\mathrm{dxdt}=\rho_i\Big\langle(\widetilde u\vert\widetilde n),\psi_i\Big\rangle_{H^{\frac{-1}{2}}(\partial Q_i);H^{\frac{1}{2}}(\partial Q_i)},
\;\forall\psi_i\in H^1(Q_i),$$
Let us now consider the bilinear form $L_i:{\mathcal{D}(\overline Q_i)}\times{\mathcal{D}(\overline Q_i)} \subset H(u,Q_i)\times H(u,Q_i)\longrightarrow\mathbf{R}$ defined for all $\varphi_i, \psi_i\in {\mathcal{D}(\overline Q_i)}$ by
$$L_i(\varphi_i,\psi_i)=\int_{Q_i}\left( \nabla_t\cdot(\widetilde u\varphi_i)\psi_i+( \widetilde u\vert\nabla_t\psi_i)\varphi_i\right)\mathrm{dxdt} + \int_{\partial Q_i^-\cup\widetilde\Gamma^-_i}\vert(\widetilde u\vert\widetilde n)\vert\varphi_i\psi_i d\widetilde\sigma$$
Accounting to theorem \ref{chen_frid} we have 
$$\vert L_i(\varphi_i,\psi_i)\vert\leq \Vert \nabla_t\cdot(\widetilde u\varphi_i) \Vert_{L^2(Q_i)}\Vert\psi_i\Vert_{L^2(Q_i)} + \Vert \nabla_t\cdot(\widetilde u\psi_i)- \nabla_t\cdot(\widetilde u)\psi_i\Vert_{L^2(Q_i)}\Vert\varphi_i\Vert_{L^2(Q_i)}$$$$+\Vert\varphi_i\Vert_{\partial Q_i^-\cup\widetilde\Gamma^-_i,(\widetilde u\vert\widetilde n)}\Vert\psi_i\Vert_{\partial Q_i^-\cup\widetilde\Gamma^-_i,(\widetilde u\vert\widetilde n)}.$$
And the following estimate holds true
$$\vert L_i(\varphi_i,\psi_i)\vert\leq (1 + \Vert \nabla\cdot u\Vert_{L^\infty(Q_i)})\Vert\varphi_i\Vert_{H(u,Q_i) }\Vert\psi_i\Vert_{H(u,Q_i) }.$$
Since it 's straigforward to check that $L_i(\varphi_i,\varphi_i)=\Vert\varphi_i\Vert^2_{\partial Q_i^+\cup\widetilde\Gamma^+_i,(\widetilde u\vert\widetilde n)}$, if we extend by continuity the bilinear form $L_i$ to $H(u,Q_i)\times H(u,Q_i)$ 
\end{proof}
\\
%

The Curved-Poincar\'e inequality for the $H(u,Q_i)-$functions is given by 
\begin{theorem}\label{poincare_loc}
Under the assumption (\ref{st16}) the semi-norm on $H(u,Q_i)$ defined by $$\vert\varphi_i\vert_{H(u,Q_i)}=\left(\int_{Q_i}(\nabla_t\cdot(\widetilde u\varphi_i))^2\mathrm{dxdt}+\int_{\partial Q_i^-\cup\widetilde\Gamma^-_i}\vert(\widetilde u\vert\widetilde n)\vert\varphi_i^2d\widetilde\sigma\right)^{\frac{1}{2}}$$
is a norm, equivalent to the norm given on $H(u,Q_i).$
\label{th6}\\
\end{theorem}

The proof is detailed in \cite{bessp,bessop} and is achieved by the inequality $\Vert\varphi_i\Vert_{L^2(Q_i)}\leq {\cal{C}}(T)\vert\varphi_i\vert_{H(u,Q_i)},$ where ${\cal{C}}(T)=\sqrt{\mathrm{max}(4T^2,2T)}$ in the divergence free advection velocity case.\\

Consider now  the least-squares formulation of the two-domain space-time problem (\ref{st11})-(\ref{st15}): find $(c^1,c^2)\in H(u,Q_1)\times H(u,Q_2)$ such that
\begin{eqnarray}
\int_{Q_1} \nabla_t\cdot(\widetilde uc^1)\nabla_t\cdot(\widetilde u\varphi^1)\;\mathrm{dxdt}&=&\int_{Q_1} f\nabla_t\cdot(\widetilde u\varphi^1)\;\mathrm{dxdt}
 \label{st21}\\
 c^1& = & c_{b,1} \qquad\qquad \mathrm{on}\;\; \partial Q^-_1
 \label{st22}\\
\int_{Q_2} \nabla_t\cdot(\widetilde uc^2)\nabla_t\cdot(\widetilde u\varphi^2)\;\mathrm{dxdt}&=&\int_{Q_2} f\nabla_t\cdot(\widetilde u\varphi^2)\;\mathrm{dxdt}
\label{st23}\\
c^2& = & c_{b,2} \qquad\qquad \mathrm{on}\;\; \partial Q^-_2
\label{st24}\\
({c^1}_{\mid_{\widetilde\Gamma^-}}-{c^2}_{\mid_{\widetilde\Gamma^-}},\mu)_{\widetilde\Gamma^-,(\widetilde u\vert\widetilde n)}&=&-({c^2}_{\mid_{\widetilde\Gamma^+}}-{c^1}_{\mid_{\widetilde\Gamma^+}},\mu)_{\widetilde\Gamma^+,(\widetilde u\vert\widetilde n)}
\label{st25}
\end{eqnarray}
$\mathrm{for}\;\mathrm{all} \;(\varphi^1,\varphi^2)\in H_0(u,Q_1,\partial Q_1^-\cup\widetilde\Gamma_1^-)\times H_0(u,Q_2,\partial Q_2^-\cup\widetilde\Gamma_2^-),$ and\\\\
$\mathrm{for}\;\mathrm{all} \;\mu\in\wedge=\{\mu\in L^2(\widetilde\Gamma^-\cup\widetilde\Gamma^+, \vert(\widetilde u\vert\widetilde n)\vert d\widetilde\sigma),\exists\; \psi^\mu\in H(u,Q),\\\hspace*{1mm}\qquad\qquad\qquad\qquad\qquad\nabla_t\cdot(\widetilde u\psi^\mu)\in L^\infty(Q),\;\psi^\mu_{\mid_{\widetilde\Gamma^+\cup\widetilde\Gamma^-}}=\mu\}\cap L^\infty(\widetilde\Gamma^-\cup\widetilde\Gamma^+).$\\
Let us show that (\ref{st21})-(\ref{st25}) can be seen as a two-domains form of the STILS problem (\ref{st19}).\\

When the flow field has a constant direction at subdomain interface, for example $(\widetilde u\vert\tilde n)$ is always positive $(\widetilde\Gamma^-=\emptyset),$ the  two-domains problem (\ref{st21})-(\ref{st25}) has a unique solution. Indeed, the equation (\ref{st21})-(\ref{st22}) can be independently solved and admits a unique solution $c^1\in H(u,Q_1)$ (theorem \ref{poincare_loc}). Next one solves (\ref{st23})-(\ref{st24}) with ${c^2}_{\mid_{\widetilde\Gamma^+}}={c^1}_{\mid_{\widetilde\Gamma^+}}.$
\\
Let us investigate the situation $\widetilde\Gamma^-\neq\emptyset,\;\widetilde\Gamma^+\neq\emptyset.$ \\
It's well known that Green's formula holds for all $\varphi_i\in H(\widetilde{\mathrm{div}},Q_i)$ and all $\psi_i\in H^1(Q_i)$ (see Lions and Magenes \cite{lions}):
$$\int_{Q_i}\psi_i(\widetilde{\mathrm{div}}\;\varphi_i)\;\mathrm{dxdt}+\int_{Q_i}(\widetilde\nabla \psi_i\vert \varphi_i)\;\mathrm{dxdt}=\int_{\partial Q_i} (\varphi_i\vert\widetilde n)\psi_id\widetilde\sigma$$
Let $\zeta_i$ be a vector valued function of $H(\widetilde{\mathrm{div}},Q_i), i=1,2,$ and $\zeta$ be the function in Q, whose restriction to $Q_i$ coincides with  $\zeta_i$. Then  $\zeta$ belongs to $H(\widetilde{\mathrm{div}},Q)$ if and only if $$(\zeta_1\vert\widetilde n)=(\zeta_2\vert\widetilde n)\;\;(H^{\frac{1}{2}}_{00}(\widetilde\Gamma))'.$$
Since the STILS solution $c\in H(u,Q),$ it's seen that $\widetilde uc\in H(\widetilde{\mathrm{div}},Q)$ and $c_{\mid_{\widetilde\Gamma}}\in L^2(\widetilde\Gamma)$. Moreover, we have $$(\widetilde uc\vert\widetilde n)\in L^2(\widetilde\Gamma^-\cup\widetilde\Gamma^+),\;\;(\widetilde uc^1\vert\widetilde n) =(\widetilde uc^2\vert\widetilde n)\quad(H^{\frac{1}{2}}_{00}(\widetilde\Gamma^-\cup\widetilde\Gamma^+))'.$$
Finally, as $H^{\frac{1}{2}}_{00}(\widetilde\Gamma^-\cup\widetilde\Gamma^+)\subset L^2(\widetilde\Gamma^-\cup\widetilde\Gamma^+)$ with dense and continuous injection see \cite{lions,Dautray}, we obtain 
$$c^1=c^2\quad a.e \quad\mathrm{on}\quad\widetilde\Gamma^-\cup\widetilde\Gamma^+$$
and (\ref{st25}) is satisfied.\\

If $c\in H(u,Q)=\overline{\mathcal{D}(\overline Q)}^{H(u,Q)}$ is the solution of (\ref{st19}), then one gets that 
${c_\mid}_{Q_i}\in H(u,Q_i)=\overline{\mathcal{D}(\overline Q_i)}^{H(u,Q_i)}.$  Under the assumptions of theorem \ref{sol_fa}, $c$ is a weak solution, and $ \nabla_t\cdot(\widetilde uc)-f =0$ in $L^2(Q)$. Let $\varphi_i\in H_0(u,Q_i,\partial Q_i^-\cup\widetilde\Gamma_i^-)$, one has $$\vert\int_{Q_i} \big{(}\nabla_t\cdot(\widetilde u{c_\mid}_{Q_i})-f\big{)}\nabla_t\cdot(\widetilde u\varphi^i)\;\mathrm{dxdt}\vert\leq\Vert \nabla_t\cdot(\widetilde uc)-f \Vert_{L^2(Q)}\Vert\nabla_t\cdot(\widetilde u\varphi^i)\Vert_{L^2(Q_i)}=0.$$ Finally,  $({c_\mid}_{Q_1},{c_\mid}_{Q_2})$ is a solution of (\ref{st21})-(\ref{st25}).
  Thus we have the following result.
  
  \begin{proposition}\label{sol2}
 Assuming $u\in L^1(0,T;BV(\Omega)^d)$ with $\nabla_t\cdot\widetilde u=0$ and $f\in L^\infty(Q)$, the STILS solution of (\ref{st19}) is a solution of the two-domains STILS problem (\ref{st21})-(\ref{st25}) i.e: If $c\in H(u,Q)$ is the solution of (\ref{st19}) then $(c_{\mid_{Q_1}},c_{\mid_{Q_2}})\in H(u,Q_1)\times H(u,Q_2)$ satisfies (\ref{st21})-(\ref{st25}). 
 \label{th7}\\\\
\end{proposition}
To get the reciprocal it suffices to show that the interface equation associated to (\ref{st25}) admits a unique solution.\\
In the sequel $\Big\langle .,.\Big\rangle_W$ denotes the duality pairing between $W$ and its topological dual $W'.$
\subsubsection{Interface equation}
 Let $\eta\in \wedge,$ and let $c^{i,\eta}\in H_0(u,Q_i,\partial Q_i^-)$ be the STILS-homogeneous extension of $\eta$ on $Q_i,\; i=1,2$ defined by
\begin{eqnarray}
\int_{Q_i} \nabla_t\cdot(\widetilde uc^{i,\eta})\nabla_t\cdot(\widetilde u\varphi^i)\;\mathrm{dxdt}
&=&0
\quad\forall\; \varphi_i\in H_0(u,Q_i,\partial Q_i^-\cup\widetilde\Gamma_i^-)
\label{st26}\\
{\gamma^-_i(c^{i,\eta})}_{\mid_{\widetilde\Gamma^-_i}}&=&\eta.
\label{st27}
\end{eqnarray}
Let also $c^{i,*}\in H_0(u,Q_i,\widetilde\Gamma_i^-)$ be the Dirichlet homogeneous extension satisfying 
\begin{eqnarray}
\int_{Q_i} \nabla_t\cdot(\widetilde uc^{i,*})\nabla_t\cdot(\widetilde u\varphi^i)\;\mathrm{dxdt}
&=&
\int_{Q_i} f\nabla_t\cdot(\widetilde u\varphi^i)\;\mathrm{dxdt} 
\label{st28}\\
{\gamma^-_i(c^{i,*})}_{\mid_{\partial Q_i^-}}&=&c_b.
\end{eqnarray}
\\ for all $\varphi^i\in H_0(u,Q_i,\partial Q_i^-\cup\widetilde\Gamma_i^-).$ 
Thanks to Lax-Milgram Lemma, the Curved-Poincar\'e inequality (theorem \ref{poincare_loc}) shows that these solutions are unique.\\
Let $(c^1,c^2)$ a solution of the problem (\ref{st21})-(\ref{st25}). If $\lambda={{c^1}_{\mid}}_{\widetilde\Gamma^-\cup\widetilde\Gamma^+}={{c^2}_{\mid}}_{\widetilde\Gamma^-\cup\widetilde\Gamma^+}\in\wedge$ then $c^i=c^{i,\lambda}+c^{i,*}$. Conversely, $(c^{1,\lambda}+c^{1,*},c^{2,\lambda}+c^{2,*})\in H(u,Q_1)\times H(u,Q_2)$ is a solution of the two-domains formulation if and only if it satisfies (\ref{st25}) : $$(\lambda-c^{2,\lambda}-c^{2,*},\mu)_{\widetilde\Gamma^-,(\widetilde u\vert\widetilde n)}+(\lambda-c^{1,\lambda}-c^{1,*},\mu)_{\widetilde\Gamma^+,(\widetilde u\vert \widetilde n)}=0,\;\forall \mu\in\wedge.$$
 Thus to obtain a solution $(c^1,c^2)$ of (\ref{st21})-(\ref{st25}) such that ${{c^i}_{\mid}}_{\widetilde\Gamma^-\cup\widetilde\Gamma^+}\in\wedge,\;i=1,2$ it is sufficient to find $\lambda\in\wedge$ satisfying the Steklov-Poincar\'e equation 
\begin{eqnarray}
  \hspace*{2mm}\qquad\Big\langle S\lambda,\mu \Big\rangle_{\wedge}= \Big\langle\chi,\mu \Big\rangle_{\wedge} \; \quad\forall\;\mu\in\wedge
\label{st29}
\end{eqnarray}
with
$S=S_1+S_2\;$  and $\chi=\chi_1+\chi_2$ such that $$\Big\langle S_i\lambda,\mu \Big\rangle_{\wedge}:=(\lambda,\mu)_{\widetilde\Gamma^-_i,(\widetilde u\vert \widetilde n)}-(c^{i,\lambda},\mu)_{\widetilde\Gamma^+_i,(\widetilde u\vert\widetilde n)},$$ and 
$$\Big\langle\chi_i,\mu \Big\rangle_{\wedge}:=(c^{i,*},\mu)_{\widetilde\Gamma^+_i,(\widetilde u\vert\widetilde n)}.$$
 Consider $\xi_i\in \mathcal{D}(\overline Q_i)\;\;\mathrm{and} \;\;\zeta_i$  a regular enough function, we have  
$$\int_{Q_i} \nabla_t\cdot(\widetilde u\xi_i\zeta_i)\mathrm{dxdt}=\int_{\partial Q_i}(\widetilde u\vert\widetilde n)\xi_i\zeta_i\mathrm{d}\widetilde\sigma=\int_{\partial Q^-_i\cup\widetilde\Gamma^-_i}(\widetilde u\vert\widetilde n)\xi_i\zeta_i\mathrm{d}\widetilde\sigma + \int_{\partial Q^+_i\cup\widetilde\Gamma^+_i}(\widetilde u\vert\widetilde n)\xi_i\zeta_i\mathrm{d}\widetilde\sigma$$
then
$$\int_{\partial Q^+_i\cup\widetilde\Gamma^+_i}(\widetilde u\vert\widetilde n)\xi_i\zeta_i\mathrm{d}\widetilde\sigma=-\int_{\partial Q^-_i\cup\widetilde\Gamma^-_i}(\widetilde u\vert\widetilde n)\xi_i\zeta_i\mathrm{d}\widetilde\sigma + \int_{Q_i} (\nabla_t\cdot(\widetilde u\xi_i)\zeta_i + (\widetilde u\vert\widetilde\nabla\zeta_i)\xi_i) \mathrm{dxdt}$$
and
\begin{eqnarray}
\nonumber
\int_{\partial Q^+_i\cup\widetilde\Gamma^+_i}(\widetilde u\vert\widetilde n)\xi_i\zeta_i\mathrm{d}\widetilde\sigma&=& -\int_{\partial Q^-_i\cup\widetilde\Gamma^-_i}(\widetilde u\vert\widetilde n)\xi_i\zeta_i\mathrm{d}\widetilde\sigma \\&&\nonumber+ \int_{Q_i} (\nabla_t\cdot(\widetilde u\xi_i)\zeta_i + \int_{Q_i} (\nabla_t\cdot(\widetilde u\zeta_i)\xi_i 
- (\nabla_t\cdot \widetilde u)\xi_i\zeta_i) \mathrm{dxdt}
\end{eqnarray}
Thus by density, if $\nabla_t\cdot\widetilde u=0$ the following result is obtained.
\begin{proposition}\label{int_bord}
Let $\varphi^i\in H(u,Q_i)$ such that $\nabla_t\cdot(\widetilde u\varphi^i)=0$ in $L^2(Q_i)$, then\\$$\int_{\partial Q_i^+\cup\widetilde\Gamma^+_i}(\widetilde u\vert\widetilde n){\varphi^i}^2\mathrm{d}\widetilde\sigma =-\int_{\partial Q_i^-\cup\widetilde\Gamma^-_i}(\widetilde u\vert\widetilde n){\varphi^i}^2\mathrm{d}\widetilde\sigma. $$
\label{th8}\\
\end{proposition}

If $u\in L^1(0,T,BV(\Omega)^d)$, $\nabla_t\cdot\widetilde u=0$, the solution $c^{i,\eta}$ of the subproblem (\ref{st26})-(\ref{st27}) is also a weak solution (theorem \ref{sol_fa}). Thus $\nabla_t\cdot (\widetilde uc^{i,\eta})=0$ in $L^2(Q_i)$. And 
the above proposition shows that $S=S_1+S_2:\wedge\longrightarrow\wedge'$ such that $S_i\eta={\eta_\mid}_{\widetilde\Gamma^-_i}-{{c^{i,\eta}}_\mid}_{\widetilde\Gamma^+_i}$ is continuous (i.e $S_i\eta\in\wedge'$).

Let $\eta\in \wedge$
we have
$$\Big\langle S_i\eta,\eta\Big\rangle_{\wedge}=(\eta,\eta)_{\widetilde\Gamma^-_i,(\widetilde u\vert\widetilde n)}-(c^{i,\eta},\eta)_{\widetilde\Gamma^+_i,(\widetilde u\vert\widetilde n)}.$$
Moreover 
\begin{eqnarray}
\frac{1}{2}\Vert c^{i,\eta}\Vert^2_{\widetilde\Gamma^+_i,(\widetilde u\vert\widetilde n)}  -(c^{i,\eta},\eta)_{\widetilde\Gamma^+_i,(\widetilde u\vert\widetilde n)}=\frac{1}{2}\Vert\eta-c^{i,\eta}\Vert^2_{\widetilde\Gamma^+_i,(\widetilde u\vert\widetilde n)}-\frac{1}{2}\Vert\eta\Vert^2_{\widetilde\Gamma^+_i,(\widetilde u\vert\widetilde n)}. 
\label{st30}
\end{eqnarray}

Finally applying relation (\ref{st30}), and proposition \ref{int_bord} to $\Big\langle S_i\eta,\eta\Big\rangle_\wedge$, the following result is obtained.
\begin{proposition}For $i=1,2$ we have\\
$\Big\langle S_i\eta,\eta\Big\rangle_{\wedge}\; \geq \frac{1}{2}(\Vert c^{i,\eta}\Vert^2_{\partial Q_i^+,(\widetilde u\vert\widetilde n)} + \Vert \eta\Vert^2_{\widetilde\Gamma^-_i ,(\widetilde u\vert\widetilde n)} + \Vert \eta-c^{i,\eta}\Vert^2_{\widetilde\Gamma^+_i,(\widetilde u\vert\widetilde n)}-\Vert \eta\Vert^2_{\widetilde\Gamma^+_i,(\widetilde u\vert\widetilde n)}),\; \forall \;\eta\in\wedge.$
\label{th9}\\

\end{proposition}
Thus for all $\eta\in\wedge$ one gets $$\Big\langle S\eta,\eta\Big\rangle_{\wedge}\; \geq \frac{1}{2}(\Vert c^{1,\eta}\Vert^2_{\partial Q_1^+,(\widetilde u\vert\widetilde n)} +\Vert c^{2,\eta}\Vert^2_{\partial Q_2^+,(\widetilde u\vert\widetilde n)}  
+ \Vert \eta-c^{1,\eta}\Vert^2_{\widetilde\Gamma^+_1,(\widetilde u\vert\widetilde n)} + \Vert \eta-c^{2,\eta}\Vert^2_{\widetilde\Gamma^+_2,(\widetilde u\vert\widetilde n)})
.$$
Consider now $\eta\in\wedge$ such that $S\eta=0$.  Define $c^\eta$ on $Q$ such that ${c^\eta}_{\mid_{Q_i}}=c^{i,\eta}.$ Then $\gamma^-(c^\eta)=0$ and $\nabla_t\cdot(\widetilde uc^\eta)=0$ in $L^2(Q)$. Applying the theorem \ref{princ_max} to the subproblem (\ref{st26})-(\ref{st27}) with $u$ regular enough then $c^\eta\in H_0(u,Q,\partial Q^-)$ and is the least-squares solution of (\ref{st9})-(\ref{st10}) with $f=0$ and $c_b=0.$ Using the Curved-Poincar\'e inequality (theorem \ref{poincare}) we have $c^\eta=0$ in $L^2(Q),\;\eta=0\in  L^2(\widetilde\Gamma^-\cup\widetilde\Gamma^+,\vert(\widetilde u\vert\widetilde n)\vert\mathrm{d}\widetilde\sigma)$, so $S$ is positive define. 
Thus the following result is proved.
\begin{proposition}
If $u\in L^2(0,T;H^1_0(\Omega)^d),\;\nabla_t\cdot\widetilde u=0$ the Stelov-Poincar\'e operator $S$ realize an isomorphism between $\wedge$ and $S(\wedge)$.
\label{th10}\\
\end{proposition}

 For $f=0$, we have ${c_\mid}_{\widetilde\Gamma^-\cup\widetilde\Gamma^+}\in\wedge$ and $\chi\in S(\wedge).$ Thus the interface equation admits a unique solution.\\
  Finally the following result is obtained.
\begin{theorem}\label{th_quiv_ddm}
 Assuming $u\in L^2(0,T;H^1_0(\Omega)^d),\;\nabla_t\cdot\widetilde u=0$ and $f=0$ in (\ref{st1}), the interface equation (\ref{st29}) associated to (\ref{st19}) admits a unique solution. Thus if $c\in H(u,Q)$ is the solution of (\ref{st19}) then $({c_\mid}_{Q_1},{c_\mid}_{Q_2})\in (H(u,Q_1)\times H(u,Q_2))\cap L^\infty(Q)$ satisfies (\ref{st21})-(\ref{st25}). Reciprocally if $(c^1,c^2)\in (H(u,Q_1)\times  H(u,Q_2))\cap L^\infty(Q)$ satisfies (\ref{st21})-(\ref{st25}) then $c$ defined on $Q$ such that ${c_\mid}_{Q_i}=c^i$ is the solution of (\ref{st19}). In this sense the two formulations are equivalent.
 \label{th11}\\
 \end{theorem}
 \begin{remark}
 This analysis can be generalized to the multi-domains case where $\overline\Omega=\bigcup_{i=1}^P\overline\Omega_i,\;\Omega_i\cap\Omega_j=\emptyset,\; \overline Q=\bigcup_{i=1}^P\overline Q_i,\;\partial Q_i\cap\partial Q_j=\widetilde\Gamma_{ij},\;Q_i=\Omega_i\times(0,T)$, and
 $c\in H(u,Q)$ is a solution of (\ref{st19}) if and only if $ ({c_{\mid}}_{Q_1},{c_{\mid}}_{Q_2},\cdots,{c_{\mid}}_{Q_P})=(c^1,c^2,\cdots,c^P)
 \in \prod_{i=1}^PH(u,Q_i)$ is a multi-domains STILS solution, i.e
 \begin{eqnarray}
\nonumber
\int_{Q_i} \nabla_t\cdot(\widetilde uc^i)\nabla_t\cdot(\widetilde u\varphi^i)\;\mathrm{dxdt}&=&\int_{Q_i} f\nabla_t\cdot(\widetilde u\varphi^i)\;\mathrm{dxdt}\;\;i=1,\cdots,P
 \nonumber\\
 c^i& = & c_{b,i} \qquad\qquad \mathrm{on}\;\; \partial Q^-_i
\nonumber \\
c^{i}&=&c^j\qquad\qquad\;\; \mathrm{on}\;\;\widetilde\Gamma_{ij}\neq\emptyset,\;j=1,\cdots,P.
\nonumber
\end{eqnarray}
for all $\varphi_i\in H_0(u,Q_i,\partial Q_i^-\cup\widetilde\Gamma^-_{ij}).$
\end{remark}
 \\ 
 
 In the sequel ${\cal{M}}$ is some constant which do not depend on $j$.
  \subsubsection{Iteration-by-subdomain scheme }
We propose now a parallel iterative procedure to solve the two-domains STILS formulation (\ref{st21})-(\ref{st25}).  Consider for $i=1,2$
$$V_i^-=\{\varphi^i\in H_0(u,Q_i,\partial Q_i^-),\;\nabla_t\cdot(\widetilde u\varphi_i)\in L^\infty(Q_i),\;{\gamma_i}^-(\varphi_i)\in L^2(\widetilde\Gamma^-_i,\vert(\widetilde u\vert\widetilde n)\vert\mathrm{d}\widetilde\sigma)\},$$
$$G_i^-={\gamma_i^-}(V_i^-).$$
For $j\geq 0,$ assuming $ ({c^{1,j}}_{\mid_{\tilde\Gamma^+}}, {c^{2,j}}_{\mid_{\tilde\Gamma^-}})\in G^-_2\times G^-_1,$
find $c^{1,j+1}$ and $c^{2,j+1}$ respectively  the least-squares solution of
\begin{eqnarray}
\nabla_t\cdot(\widetilde uc^{1,j+1})&=&f
\label{st31}\\
c^{1,j+1} &=&c_{b,1}\qquad\mathrm{for}\;(x,t)\;\;\mathrm{on}\;\partial Q^-_1
\label{st32}\\
c^{1,j+1}&=&c^{2,j}
\qquad \mathrm{for}\;(x,t)\;\;\mathrm{on}\;\tilde\Gamma^-
\label{st33}
\end{eqnarray}
and 
\begin{eqnarray}
\nabla_t\cdot(\widetilde uc^{2,j+1})&=&f
\label{st34}\\
c^{2,j+1} &=&c_{b,2}\qquad 
\mathrm{for}\;(x,t)\;\;\mathrm{on}\;\partial Q^-_2 
\label{st35}\\
c^{2,j+1}&=&c^{1,j}
\qquad 
\mathrm{for}\;(x,t)\;\;\mathrm{on}\;\tilde\Gamma^+.
\label{st36}
\end{eqnarray}
In the STILS sense these subproblems admit unique solutions. 
To analyze the convergence of this substructuring algorithm, let us remark that the subproblem (\ref{st31})-(\ref{st33}) can be decomposed into STILS-homogeneous and Dirichlet-homogeneous parts
$$c^{1,j+1}=c^{1,j+1,{\lambda^j_2}}+c^{1,*},\;\;\lambda_2^j={{c^{2,j}}_\mid}_{\widetilde\Gamma^-}$$
where the STILS-homogeneous part $c^{1,j+1,{\lambda^j_2}}\in H_0(u,Q_1,\partial Q^-_1)$ is defined as the least-squares solution of 
\begin{eqnarray}
\nabla_t\cdot(\widetilde uc^{1,j+1,{\lambda^j_2}})&=&0
\label{st37}\\
c^{1,j+1,{\lambda^j_2}} &=&0\qquad
\mathrm{for}\;(x,t)\;\;\mathrm{on}\;\partial Q^-_1
\label{st38}\\
c^{1,j+1,{\lambda^j_2}}&=&c^{2,j}\;\quad
\mathrm{for}\;(x,t)\;\;\mathrm{on}\; \tilde\Gamma^-
\label{st39}
\end{eqnarray}
and the Dirichlet-homogeneous part $c^{1,*}\in H_0(u,Q_1,\tilde\Gamma^-_1) $ is the least-squares solution of 
\begin{eqnarray}
\nabla_t\cdot(\widetilde uc^{1,*})&=&f
\label{st40}\\
c^{1,*} &=&c_{b,1} \qquad
\mathrm{for}\;(x,t)\;\;\mathrm{on}\;\partial Q^-_1
\label{st41}\\
c^{1,*}&=&0\qquad
\quad
\mathrm{for}\;(x,t)\;\;\mathrm{on}\; \tilde\Gamma^-.
\label{st42}
\end{eqnarray}
Assuming $u\in L^1(0,T;BV(\Omega)^d),\;\nabla_t\cdot\widetilde u=0$, we have $\nabla_t\cdot(\widetilde uc_1^{j+1,{\lambda^{j}_2}})=0$ in $L^2(Q_1)$ (theorem \ref{sol_fa})
and 
$$\Vert\nabla_t\cdot(\widetilde uc^{1,j+1,\lambda^j_2})\Vert^2_{L^2(Q_1)}+ \Vert c^{1,j+1,\lambda^j_2}\Vert_{\widetilde\Gamma^-,(\widetilde u\vert\widetilde n)}^2\leq \Vert c^{2,j}\Vert_{\widetilde\Gamma^-,(\widetilde u\vert\widetilde n)}^2.$$
By the trace inequality (proposition \ref{trace_loc}) , the Curved-Poincar\'e inequality (theorem \ref{poincare_loc}) and the  proposition \ref{int_bord}
$$\Vert c^{1,j+1,\lambda^j_2}\Vert^2_{\widetilde\Gamma^+,(\widetilde u\vert\widetilde n)}+ \Vert c^{1,j+1,\lambda^j_2}\Vert^2_{H(u,Q_1)}\leq ({\cal{M}}(T))^{-1}\Vert c^{2,0}\Vert_{\widetilde\Gamma^-,(\widetilde u\vert\widetilde n)}^2,$$$\mathrm{with}\;
{\cal{M}}(T)=\frac{1}{4}\mathrm{min}({\cal{C}}(T),1).$\\

In the same way, there exists a unique least-squares solution of (\ref{st34})-(\ref{st36}) given by $c^{2,j+1}=c^{2,j+1,{\lambda^j_1}}+c_2^*\in  H(u,Q_2)$ such that
$$\Vert c^{2,j+1,\lambda^j_1}\Vert^2_{\widetilde\Gamma^-,(\widetilde u\vert\widetilde n)}+ \Vert c^{2,j+1,\lambda^j_1}\Vert^2_{H(u,Q_2)}\leq ({\cal{M}}(T))^{-1}\Vert c^{1,0}\Vert_{\widetilde\Gamma^+,(\widetilde u\vert\widetilde n)}^2.$$
Thus the following estimate result is proved.
 \begin{proposition}\label{estim_iter}
 There exists a constant ${\cal{M}}>0$ depending on $T$ such that
$\Vert c^{i,j+1}\Vert^2_{\widetilde\Gamma^+_i,(\widetilde u\vert\widetilde n)}+ \Vert c^{i,j+1}\Vert^2_{H(u,Q_i)}\leq {\cal{M}},\;i=1,2\;\mathrm{for}\;\mathrm{all}\;j\geq 0$.
\label{th12}
\end{proposition}
\\

Therefore there exists a subsequence of $c^{1,j}$ (still denoted by this symbol) converging to some $c^1$ weakly in $H(u,Q_1)$. The same argument is applied to the sequence $c^{2,j}$. This allow us to pass to the limit as $j$ tends to infinity in the STILS forms of (\ref{st31})-(\ref{st36}). This limit is a solution of the two-domain STILS formulation (\ref{st21})-(\ref{st25}).\\
Thus the following convergence result is proved.
\begin{theorem}\label{th_conver}
Under the assumptions of the proposition \ref{sol2}, the sequence $(c^{1,j},c^{2,j})$ defined by the substructuring algorithm (\ref{st31})-(\ref{st36}) admits a subsequence which weakly converges in $H(u,Q_1)\times H(u,Q_2)$ to $(c^1,c^2)$
which satisfies the two-domain STILS form (\ref{st21})-(\ref{st25}).
\label{th13}
\end{theorem}
\\
\begin{remark} 
If $\widetilde\Gamma^-=\emptyset$ or $\widetilde\Gamma^+=\emptyset,$  
the substructuring iterative scheme (\ref{st31})-(\ref{st36}) converges in one iteration.
\end{remark}
\section{Time-marching approach} 
To simplify the notations we suppose $c_{00}=0,$ the general case  being obtained by superposition. Let now decompose the time domain $(0,T)$ into M parts
$$t_0=0<t_1<\cdots\cdots<t_M=T$$
and put $$\overline Q=\bigcup_{k=1}^M\overline Q^k,\;\; Q^k=\Omega\times (t_{k-1},t_k)$$
In the space-time description the flow has a constant direction at the interface $\partial Q^{k-1}\cap \partial Q^k=\Omega\times\{t_{k-1}\}$ then with the same assumptions of proposition \ref{sol2} it's seen that $c\in H(u,Q)$ is the solution of (\ref{st19}) if and only if for $k=1,\cdots,M, \;c_{k}={c_\mid}_{Q^k}\in H(u,Q^{k})$  is the unique solution of 
\begin{eqnarray}
\int_{Q^k}\nabla_t\cdot(\widetilde uc_k)\nabla_t\cdot(\widetilde u\varphi_k)\;\mathrm{dtdx} &=&\int_{Q^k} f\nabla_t\cdot(\widetilde u\varphi_k)\;\mathrm{dtdx}
\label{st43}\\
c_k(x,t_k)&=&c_{k-1}\;\;\;\mathrm{for}\;x\;\mathrm{on}\;\Omega
\label{st44}
\end{eqnarray}
for all $\varphi_k\in H_0(u,Q^k,\partial Q^{k,-}).$\\

In the same way $(c^1,c^2)\in H(u,Q^1)\times H(u,Q^2)$ is the solution of (\ref{st21})-(\ref{st25}) if and only if 
 for $k=1,\cdots,M,\;(c^1_{k},c^2_{k})=({c^1_\mid}_{Q_1^k},{c^2_\mid}_{Q_2^k})\in H(u,Q^k_1)\times H(u,Q^k_2)$ is the unique solution of
\begin{eqnarray}
\int_{Q^k_1}\nabla_t\cdot(\widetilde uc^1_k)\nabla_t\cdot(\widetilde u\varphi^1_k)\;\mathrm{dtdx} &=&\int_{Q^k_1} f\nabla_t\cdot(\widetilde u\varphi^1_k)\;\mathrm{dtdx}
\label{st45}\\
c^1_k(x,t_k)&=&c^1_{k-1}\;\;\;\mathrm{for}\;x\;\mathrm{on}\;\Omega_1
\label{st46}
\\
{{c^1_k}_\mid}_{\widetilde\Gamma^k}&=&{{c^2_k}_\mid}_{\widetilde\Gamma^k}
\label{st47}\\
\int_{Q^k_2}\nabla_t\cdot(\widetilde uc^2_k)\nabla_t\cdot(\widetilde u\varphi^2_k)\;\mathrm{dtdx} &=&\int_{Q^k_2} f\nabla_t\cdot(\widetilde u\varphi^2_k)\;\mathrm{dtdx}
\label{st48}\\
c^2_{k}(x,t_k)&=&c^2_{k-1}\;\;\;\mathrm{for}\;x\;\mathrm{on}\;\Omega_2
\label{st49}
\end{eqnarray}
for all $(\varphi^1_k,\varphi^2_k)\in H_0(u,Q^k_1,\partial Q^{k,-}_1)\times H_0(u,Q^k_2,\partial Q^{k,-}_2)$.\\

Where
$$\overline Q^k=\overline Q^k_1\cup\overline Q^k_2,\;\overline Q^k_i=\Omega_i\times(t_{k-1},t_k),\;i=1,2,\;\mathrm{and}\;\widetilde\Gamma^{k}=\partial Q_1^k\cap\partial Q_2^k$$

Using theorem \ref{th_quiv_ddm}, one shows that 
the multi-domains forms (\ref{st43})-(\ref{st44}) and (\ref{st45})-(\ref{st49}) are equivalent. Moreover for $k=1,\cdots,M$, the solution of (\ref{st45})-(\ref{st49}) can be obtained by the iterative scheme (\ref{st31})-(\ref{st36}).

In the sequel we show how to separate the two dimensions space and time to solve the above multi-domain problems with finite element discretizations. 


\section{Finite element discretizations in time and space} In this section we assume that $\nabla_t\cdot\widetilde u=0,$ and the equation (\ref{st9}) becomes
    $$(\nabla_tc\vert\widetilde u)=f.$$ 
Then the different formulations 
(\ref{st19}) and (\ref{st21})-(\ref{st25}) are now performed by separated finite element discretizations in the following way. 
\\
 
For the space dimension, we consider $\{\mathcal{P}_h(\Omega)\}_h$ be a family of regular and quasi-uniform triangulations of $\Omega,$ i.e satisfying the following conditions: there are two positive constant $\alpha$ and $\sigma$ independent of $h={\max}_{\mathcal{K}\in\mathcal{P}_h(\Omega)} h_{\mathcal{K}}$ such that for all $\mathcal{K}\in\mathcal{P}_h(\Omega)$\\
$(i)\;\;h_{\mathcal{K}}\geq\sigma h$\\
$(ii)\;$the angles of $\mathcal{K}$ are bounded from below by $\alpha$\\
\\
where $h_{\mathcal{K}}$ is the diameter of $\mathcal{K}$.
This triangulation can be induced by two independent grids $\mathcal{P}_{1h}$ and $\mathcal{P}_{2h}$ defined on the subdomains $\Omega_1$ and $\Omega_2$, and which are compatible on $\Gamma$ (that is they share the same edges therein). \\
Let $\{\psi_1\cdots\psi_N\}$ be the basis of a finite dimensioned subspace $\mathcal{V}_h\subset H(u,Q),$ where $\psi_i$ is the polynomial function which interpolates the pairs \\$(x_1,0),(x_2,0),\cdots,(x_i,1),\cdots,(x_N,0), $ ($x_1,x_2,\cdots,x_i,\cdots,x_N$ are the nodes of the grid $\{\mathcal{P}_h(\Omega)\}_h$).  \\

For the time dimension, the domain $[0,T]$ is divided  into $n$ subintervals \\$[t^0,t^1],\;[t^1,t^2],\;\cdots,[t^{n-1},t^n]$ as follows
$$0=t^0<t^1<\cdots<t^{n-1}<t^n=T$$ and for each $(t^{k-1},t^k),\;k=1,\cdots,n$ we consider affine basis functions
$$a_{k-1}(t)=\frac{1}{\tau}(t^{k}-t),\;\;a_{k}(t)=\frac{1}{\tau}(t-t^{k-1})$$
with $t\in I^k=[t^{k-1},t^{k}],\;\tau=t^{k}-t^{k-1}.$ Let $C_2$ be the $2-$dimensional subspace of $\mathcal{C}[t^{k-1},t^{k}]$ generated by $a_{k-1}$ and $a_{k}$.\\

Put now $$Q^{I^k}=\Omega\times I^k,\;Q^{I^k}_i=\Omega_i\times I^k,\;i=1,2$$
Then the TMQ1 (Time-Marching Method with Q1 basis functions in time) approximation \cite{besg} of the problem (\ref{st19}) 
consists in finding for $k\ge 1,\;c_h=c^{k}\in H^h=(\mathcal{V}^h\otimes C_2)\cap H(u,Q^{I^k})$ such that
\begin{eqnarray}
\int_{Q^{I^k}}(\nabla_tc_h\vert\widetilde u)(\nabla_t\varphi_h\vert\widetilde u)
\;\mathrm{dtdx} &=&\int_{Q^{I^k}} f (\nabla_t\varphi_h\vert\widetilde u)\;\mathrm{dtdx}
\label{st50}\\
c_h(x,t^{k-1})&=&c^{k-1}\;\;\;\mathrm{for}\;x\;\mathrm{on}\;\Omega
\label{st51}
\end{eqnarray}
for all $\varphi_h\in H^{h,0}=(\mathcal{V}^h\otimes C_2)\cap H_0(u,Q,\partial Q^{I^k,-}).$\\

Consider $\{\psi^{(i)}_1,\cdots,\psi^{(i)}_{N^i}\}$ the set of the basis functions associated with the nodes lying on $\Omega_i,\;i=1,2.$ And $\mathcal{V}^h_{i}$ be the finite dimensioned subspace induced by this set of functions. 
Then the
TMQ1 approximation 
of (\ref{st21})-(\ref{st25}) is expressed as follows:
for $k\ge 1,\;(c^1_h,c^2_h)=(c_1^{k},c_2^{k})\in H_1^h\times H_2^h=\big((\mathcal{V}^h_1\otimes C_2)\cap H(u,Q^{I^k}_1)\big)\times \big((\mathcal{V}^h_2\otimes C_2)\cap H(u,Q^{I^k}_2)\big)$ such that
\begin{eqnarray}
\int_{Q^{I^k}_1}(\nabla_tc^1_h\vert\widetilde u)(\nabla_t\varphi^1_h\vert\widetilde u)
\;\mathrm{dtdx} &=&\int_{Q^{I^k}_1} f(\nabla_t\varphi^1_h\vert\widetilde u)\;\mathrm{dtdx}
\label{st52}\\
c^1_h(x,t^{k-1})&=&c_1^{k-1}\;\;\;\mathrm{for}\;x\;\mathrm{on}\;\Omega_1
\label{st53}
\\
{{c^1_h}_\mid}_{\widetilde\Gamma^{I^k}}&=&{{c^2_h}_\mid}_{\widetilde\Gamma^{I^k}}
\label{st54}\\
\int_{Q^{I^k}_2}(\nabla_tc^2_h\vert\widetilde u)(\nabla_t\varphi^2_h\vert\widetilde u)
\;\mathrm{dtdx} &=&\int_{Q^{I^k}_2} f(\nabla_t\varphi^2_h\vert\widetilde u)\;\mathrm{dtdx}
\label{st55}\\
c^2_h(x,t^{k-1})&=&c_2^{k-1}\;\;\;\mathrm{for}\;x\;\mathrm{on}\;\Omega_2
\label{st56}
\end{eqnarray}
$(\varphi^1_h,\varphi^2_h)\in H_1^{h,0}\times H_2^{h,0}=\big((\mathcal{V}^h_1\otimes C_2)\cap H_0(u,Q^{I^k}_1,\partial Q^{I^k,-}_1\cup\widetilde\Gamma^{I^k,-})\big)\times \big((\mathcal{V}^h_2\otimes C_2)\cap H_0(u,Q^{I^k}_2,\partial Q^{I^k,-}_2\cup\widetilde\Gamma^{I^k,+})\big)$.\\
where
$\widetilde\Gamma^{I^k,-}$ and $\widetilde\Gamma^{I^k,+}$ are respectively the inflow and outflow part of $\widetilde\Gamma^{I^k}.$ \\
For each time subdomain $(t^{k-1},t^k),$ the equivalent result between the discrete problems (\ref{st50})-(\ref{st51}) and (\ref{st52})-(\ref{st56}) is given by

\begin{theorem}\label{equi_disc}
 Assuming $\nabla_t\cdot\widetilde u=0,$ and $f=0$ in (\ref{st1}), 
 then if 
 $c_h\in H^h$ is a solution of (\ref{st50})-(\ref{st51}) then $({{c_h}_\mid}_{Q_1},{{c_h}_\mid}_{Q_2})\in (H_1^{h}\times H_2^{h})$ satisfies (\ref{st52})-(\ref{st56}). Reciprocally if $(c_h^1,c^2_h)\in H_1^h\times H_2^h$ satisfies (\ref{st52})-(\ref{st56}) then $c_h$ defined on $Q^k$ such that ${{c_h}_\mid}_{Q^{I^k}_i}=c^i_h$ is a solution of (\ref{st50})-(\ref{st51})\label{th14}.
\end{theorem}
\\

Finally as in the continuous level (section 5) the above TMQ1 approximations are equivalent.\\

In order to solve the two-domains discrete formulation (\ref{st52})-(\ref{st56}), it's convenient to apply 
the iteration-by-subdomain scheme.\\
Thus for $j\geq 0$ until convergence,  
find ${c^1_h}^{,j+1}=c^{k,j+1}_1\in H_1^h$ and ${c^2_h}^{,j+1}=c^{k,j+1}_2\in H_2^h$ such that
 \begin{eqnarray}
\int_{Q^{I^k}_1} (\nabla_tc^{1,j+1}_h\vert\widetilde u)(\nabla_t\varphi^1_h\vert\widetilde u)&=&\int_{Q^{I^k}_1}f(\nabla_t\varphi^1_h\vert\widetilde u),\;\mathrm{for}\;\mathrm{all}\;\varphi^1_h\in H_1^{h,0}
\label{st57}\\
c_h^{1,j+1}(x,t^{k-1}) &=&c^{k-1}_{1}
\label{st58}\\
{{c_h^{1,j+1}}_{\mid}}_{\widetilde\Gamma^{I^k,-}}&=&c_h^{2,j}
\label{st59}
\end{eqnarray}
and 
\begin{eqnarray}
\int_{Q^{I^k}_2} (\nabla_tc^{2,j+1}_h\vert\widetilde u)(\nabla_t\varphi^2_h\vert\widetilde u)&=&\int_{Q^{I^k}_2}f(\nabla_t\varphi^2_h\vert\widetilde u),\;\mathrm{for}\;\mathrm{all}\;\varphi^2_h\in H_2^{h,0}
\label{st60}\\
c_h^{2,j+1} (x,t^{k-1})&=&c_2^{k-1}
\label{st61}\\
{{c_h^{2,j+1}}_{\mid}}_{\widetilde\Gamma^{I^k,+}}&=&c_h^{1,j}
\label{st62}
\end{eqnarray}
Following the same guideline of the continuous level, we get the estimates
$$\Vert c^{i,j}_h\Vert^2_{\tilde\Gamma^{k,+}_i,(\tilde u\vert\tilde n)}+ \Vert c^{i,j}_h\Vert^2_{H(u,Q_i^{I^k})}\leq {\cal{M}}(\tau),\;\;i=1,2,
$$
and the discrete convergence result is given by
\begin{theorem}\label{conver_dis}
Under the assumptions of the proposition \ref{sol2}, the sequence $(c_h^{1,j},c_h^{2,j})$ defined by the substructuring algorithm (\ref{st57})-(\ref{st62}) admits a subsequence which weakly converges in $H_1^h \times H_2^h$ to $(c_h^1,c_h^2)$
which satisfies the discrete two-domains STILS form (\ref{st52})-(\ref{st56})\label{th15}.\\
\end{theorem}
\\
Higher orders time basis functions can be considered in the above analysis. For example one can to apply the time-marching approach with $Q2-$basis functions \cite{besg} (TMQ2) which are defined on $[0,\tau]$ by
\begin{eqnarray}
\nonumber
a_{k-1}(t)&=&\frac{(t-\frac{\tau}{2})(t-\tau)}{(0-\frac{\tau}{2})(0-\tau)} \\\nonumber
a_{k-\frac{1}{2}}(t)&=&\frac{(t-0)(t-\tau)}{(\frac{\tau}{2}-0)(\frac{\tau}{2}-\tau)}\\\nonumber
a_{k}(t)&=&\frac{(t-0)(t-\frac{\tau}{2})}{(\tau-\frac{\tau}{2})(\tau-\frac{\tau}{2})}
\end{eqnarray}
\begin{remark}
The subproblems (\ref{st57})-(\ref{st59}) and (\ref{st60})-(\ref{st62}) can be approximated differently both in space and time. 
 Indeed: \\
 $1.\;$ For the space domain, the two grids $\mathcal{P}_{1h}$ and $\mathcal{P}_{2h}$ can be choosed independently.\\
 $2.\;$ For the time subdomain $(t^{k-1},t^{k})$, the choice of the basis functions can also be different for the two subproblems, for example one can to adopt the TMQ1 method for 
 (\ref{st57})-(\ref{st59}), and the TMQ2 approximation for 
 (\ref{st60})-(\ref{st62}). 
\end{remark}
\\

In the next section some numerical results are presented to analyze the potential of the iterative scheme (\ref{st57})-(\ref{st62}) to solve the STILS-DDM formulation (\ref{st52})-(\ref{st56}). In this sense the STILS-DDM numerical solution 
and the STILS numerical solution 
are compared for the cylinder rotating example \cite{hans}. 
\section{Numerical results}
We consider the TMQ1 approach to analyze numerically the STILS and STILS-DDM formulations
\subsection{Time integration}
To solve (\ref{st50})-(\ref{st51}) by a TMQ1 method, we choose $\varphi_h$ independent of $a_{k-1}$. So $\varphi_h$ cancels on $t=t^{k-1}$. This is motivated by considering the local problem in the domain $Q^{I^k}=\Omega\times I^k$, with an initial condition (\ref{st51}) on $\Omega\times\{t^{k-1}\}.$ If 
     $$c^k(x,t)=\sum_{l=1}^N\psi_l(x)(c^{k-1}_{l}a_{k-1}(t)+c^{k}_{l}a_{k}(t)),$$
problem (\ref{st50}) becomes in this case, after time integration see \cite{besg}
\begin{eqnarray}
\label{st63}&&\sum_{l=1}^Nc^{k}_{l}\int_{\Omega}\Big[\frac{\tau}{3}(\nabla\psi_l\vert u)(\nabla\psi_i\vert u)+ \frac{1}{2}(\nabla\psi_l\vert u)\psi_i+ \frac{1}{2}(\nabla\psi_i\vert u)\psi_l+ \frac{1}{\tau}\psi_l\psi_i\Big]\mathrm{dx}\\\nonumber\\&=&\sum_{l=1}^Nc^{k-1}_{l}\int_{\Omega}\Big[\frac{-\tau}{6}(\nabla\psi_l\vert u)(\nabla\psi_i\vert u) -\frac{1}{2}(\nabla\psi_l\vert u)\psi_i+ \frac{1}{2}(\nabla\psi_i\vert u)\psi_l+ \frac{1}{\tau}\psi_l\psi_i\Big]\mathrm{dx}\;\;\nonumber
\end{eqnarray}
for all $i$ in $\{1,\cdots,N\}$.
For the STILS-DDM algorithm (\ref{st57})-(\ref{st62}), we choose $\varphi^i_h,\;i=1,2$ independent of $a_{k-1}$, and if 
$$c^{k,j+1}_m(x,t)=\sum_{l=1}^{N^m}\psi^{(m)}_l(x)(c^{k-1}_{m,l}a_{k-1}(t)+c^{k,j+1}_{m,l}a_{k}(t)),\;\;m=1,2,\;j\geq 0.$$
then (\ref{st57})-(\ref{st59}) and (\ref{st60})-(\ref{st62}) become
\begin{eqnarray}
\label{st64}&&\qquad\sum_{l=1}^{N^m}c^{k,j+1}_{m,l}\int_{\Omega_m}\Big[\frac{\tau}{3}(\nabla\psi^{(m)}_l\vert u)(\nabla\psi^{(m)}_i\vert u)+ \frac{1}{2}(\nabla\psi^{(m)}_l\vert u)\psi^{(m)}_i+\\\nonumber\\\nonumber
&& \frac{1}{2}(\nabla\psi^{(m)}_i\vert u)\psi^{(m)}_l+ \frac{1}{\tau}\psi^{(m)}_l\psi^{(m)}_i\Big]\mathrm{dx}
=\sum_{l=1}^{N^m}c^{k-1}_{m,l}\int_{\Omega_m}\Big[\frac{-\tau}{6}(\nabla\psi^{(m)}_l\vert u)(\nabla\psi^{(m)}_i\vert u) - \\\nonumber\\
&&\frac{1}{2}(\nabla\psi^{(m)}_l\vert u)\psi^{(m)}_i+\frac{1}{2}(\nabla\psi^{(m)}_i\vert u)\psi^{(m)}_l+ \frac{1}{\tau}\psi^{(m)}_l\psi^{(m)}_i\Big]\mathrm{dx},
\nonumber
\end{eqnarray}
$m=1,2,$ for all $i$ in $\{1,\cdots,N\}$, and $j\geq 0.$ \\

The related sparse, symmetric, and positive define systems 
(\ref{st63}), (\ref{st64}) are solved by means of the Preconditioned Conjugate Gradient  (PCG) algorithm.
\subsection{Cylinder with a slot example}
 Let consider the cylinder with a slot example taken from \cite{hans}. The domain is the square $\Omega=]-1,1[^2$ and discretized in $100\;\times\;100$ elements, of type $Q1$. 
 The initial condition is  $$c(x,y,0)=
  \left\{
  \begin{tabular}{ll}
  $1 \;\mathrm{if}\;(\vert x\vert>0.05\;\mathrm{or}\;y>0.7)\;\;\mathrm{and}\;\;R\le 3$\\
  $0 \; \mathrm{elsewhere}$\\
  \end{tabular}
 \right.\;
$$
 where $R=\sqrt{x^2+(y-0.5)^2}$, the velocity field has the form $$v(x,y,t)=(-y,x),$$ and the final time is $T=2\pi.$\\
 At final time, i.e. after a complete rotation, the exact solution coincides with the initial condition. 
 The domain is decomposed by two subdomains $$\Omega_1=]-1,0[\;\times\;]-1,1[,\,\Omega_2=]0,1[\;\times\;]-1,1[,\;\Gamma=\{0\}\times]-1,1[.$$
 Here are the results for numerical solutions of equation (\ref{st1}) using the STILS and STILS-DDM methods. 
Thus for $\tau=2\pi/800,\;h=2/100$, the figures \ref{hans1}, \ref{hans2}, and \ref{hans3} show respectively the snapshot of the numerical solutions at $100 th,\,400 th,\,500 th$ step on $800$. 
The cylinder rotating effectuates perfectly. In more the shape of the two slots are similar, and they are preserved compared to the initial condition. 
Thus this example of simulation justifies well that  the time-marching approch of the STILS-DDM procedure even its simplest form (Q1 time basis functions ) is well encouraged to approximate the STILS problem. 
\begin{figure}[H]
\begin{center}
\includegraphics[width=14.cm]{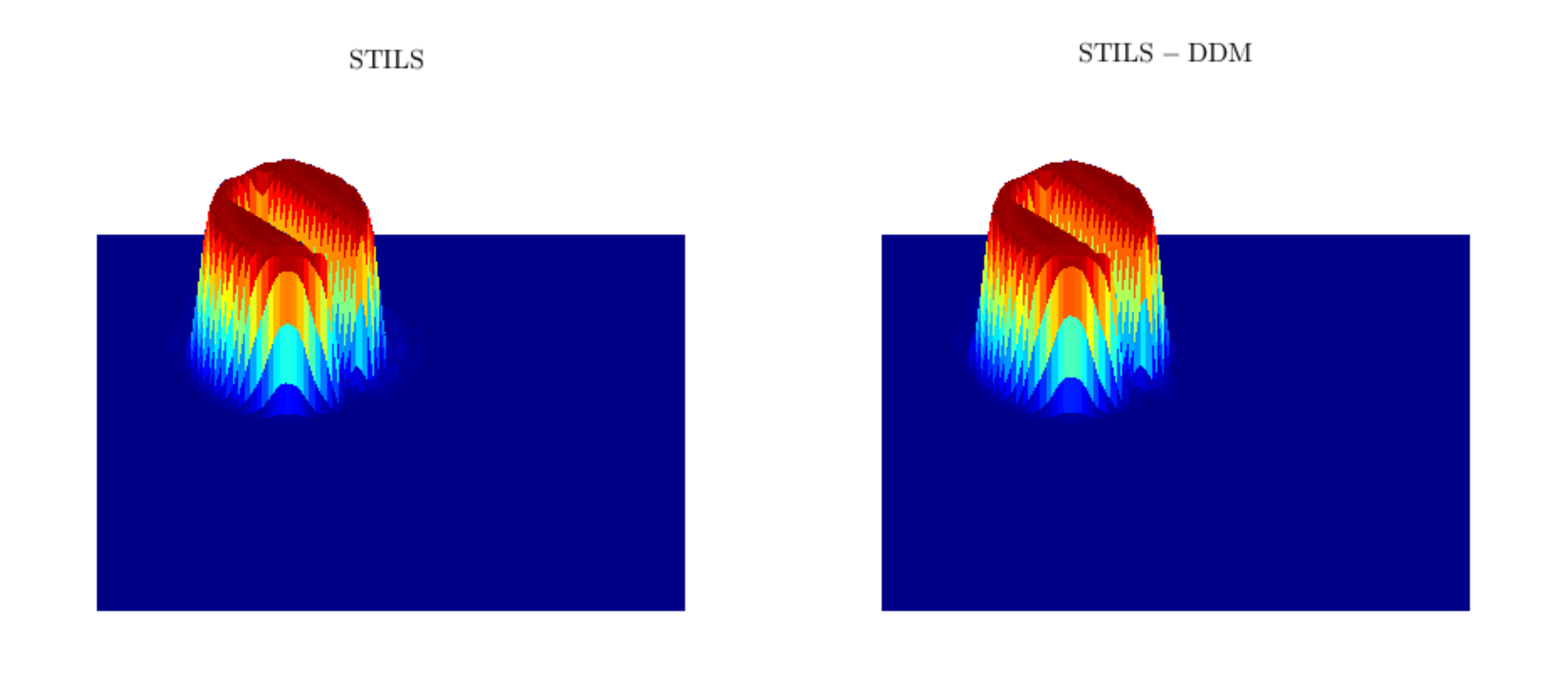}
\caption{TMQ1 numerical solutions of STILS (left), STILS-DDM (right)  at $100 th$ step}
\label{hans1}
\end{center}
\end{figure}
\begin{figure}[H]
\begin{center}
\includegraphics[width=14.cm]{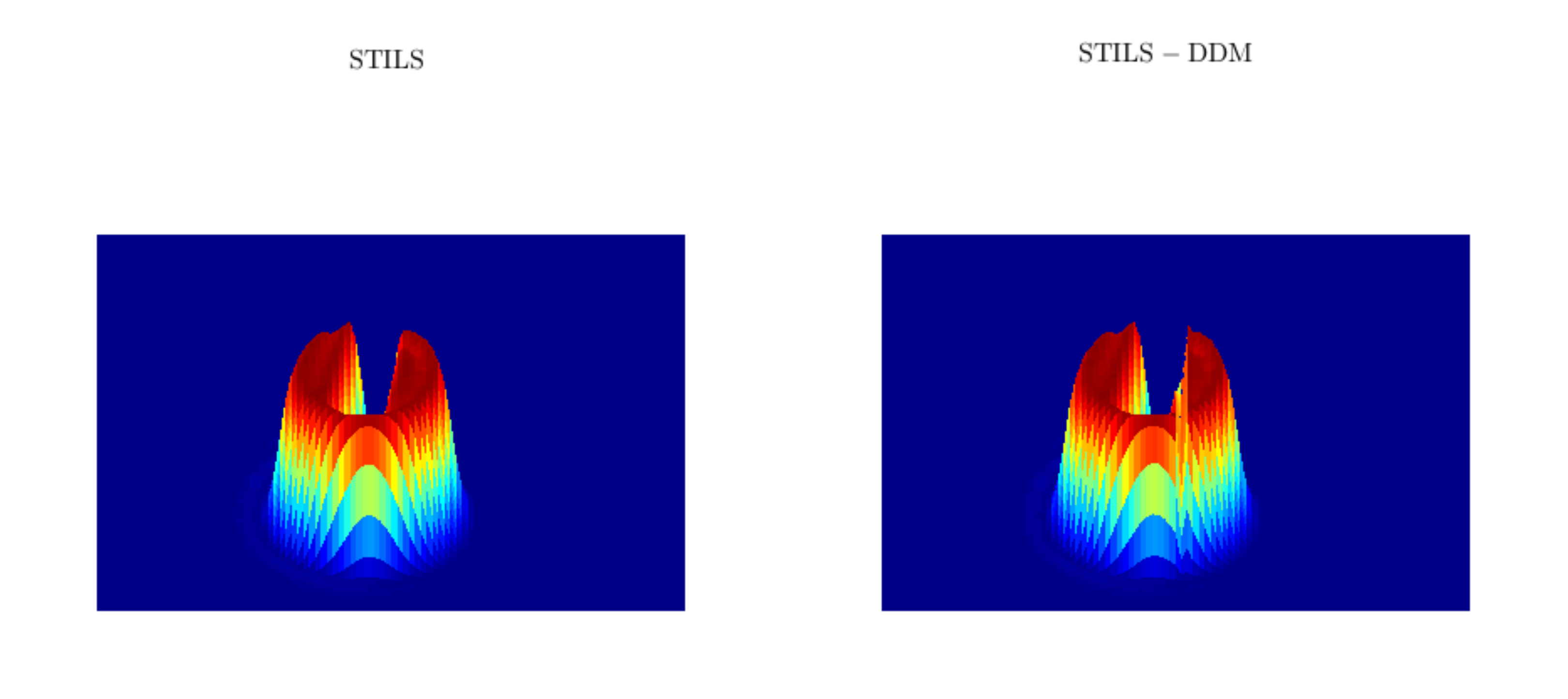}
\caption{TMQ1 numerical solutions of STILS (left), STILS-DDM (right)  at $400 th$ step}
\label{hans2}
\quad\\\quad\\\quad\\\quad\\\quad\\\quad\\\quad\\\quad\\\quad\\\quad\\\quad
\end{center}
\end{figure}
\begin{figure}[H]
\begin{center}
\includegraphics[width=14.cm]{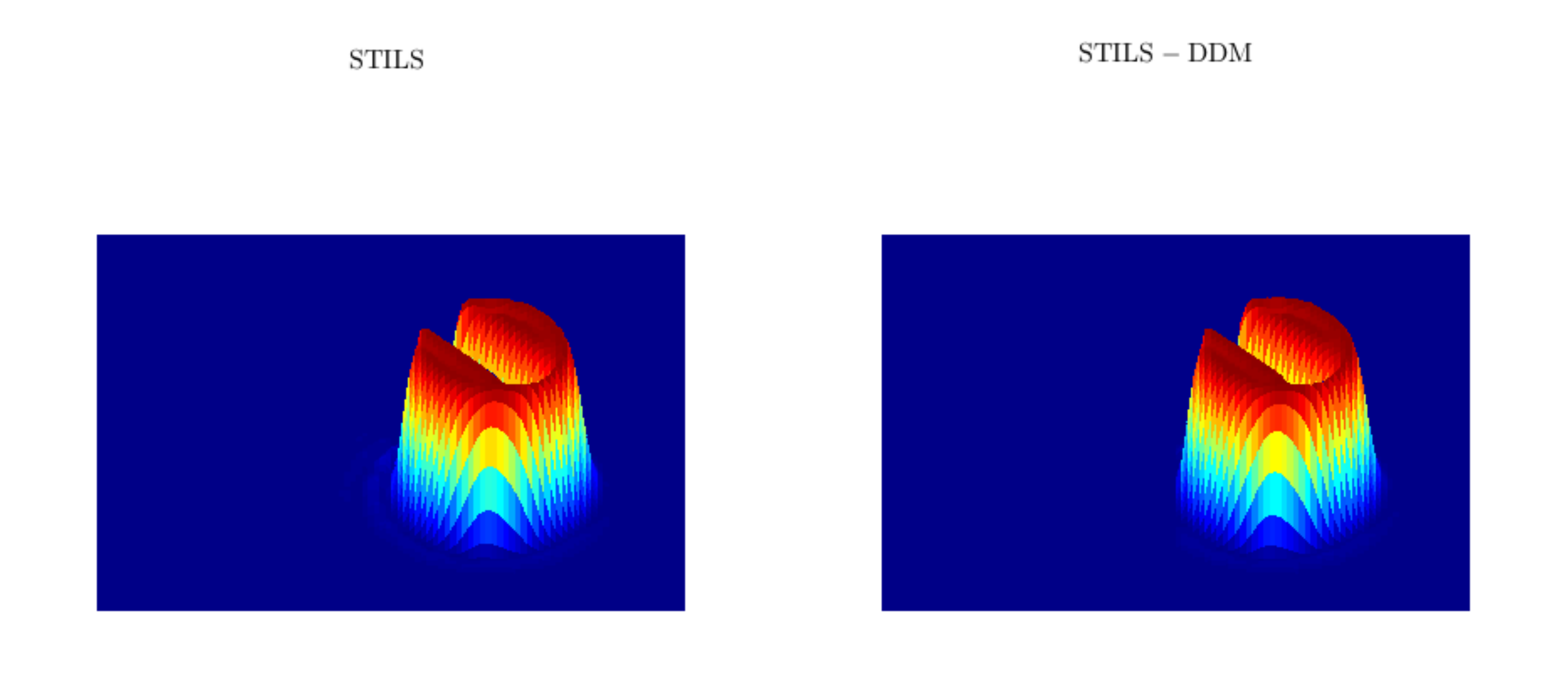}
\caption{TMQ1 numerical solutions of STILS (left), STILS-DDM (right)  at $500 th$ step}
\label{hans3}
\end{center}
\end{figure}
\begin{figure}[H]
\begin{center}
\includegraphics[width=14.cm]{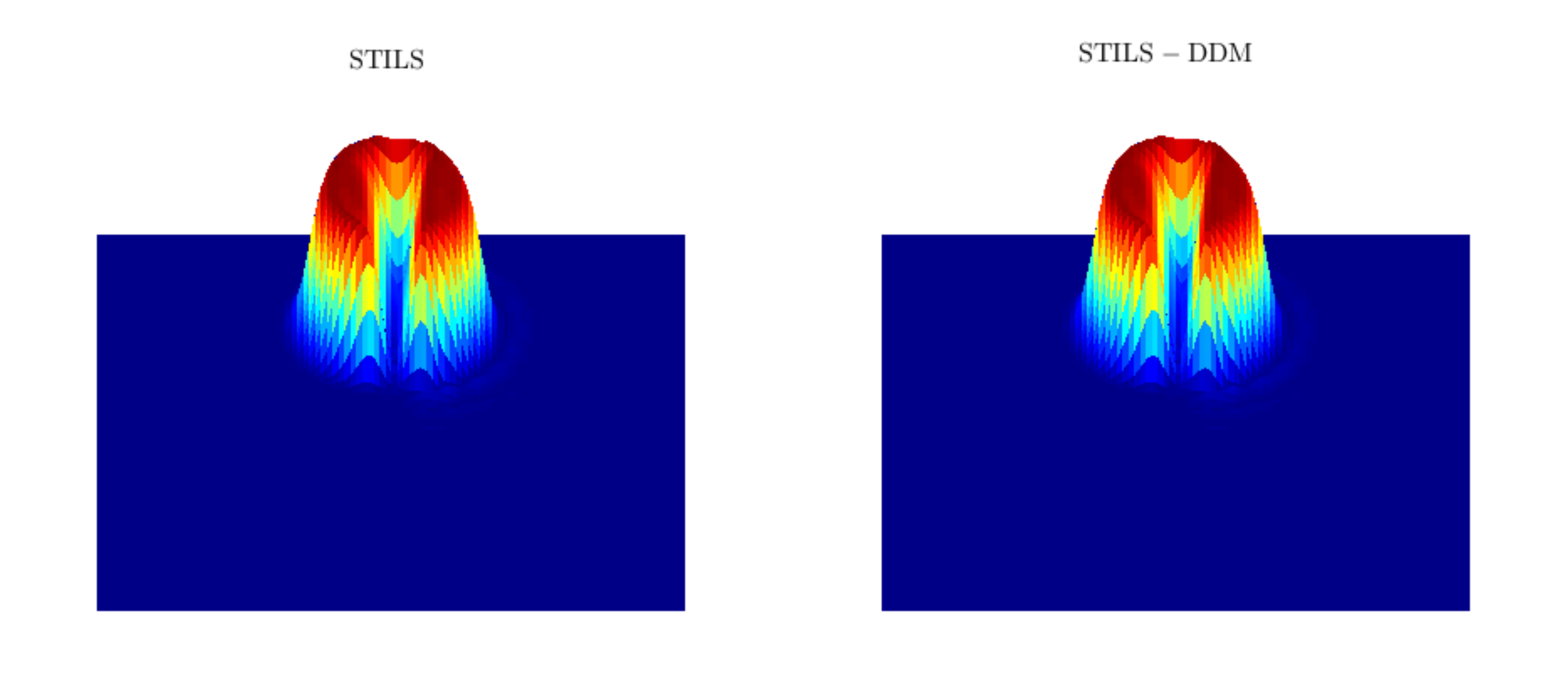}
\caption{TMQ1 numerical solutions of STILS (left), STILS-DDM (right) at  final time}
\label{hans4}
\end{center}
\end{figure}
\begin{remark}
 The STILS-DDM solution presents some small oscillations on the interface which  doesn't appear in the STILS numerical solution see fig. \ref{hans2}. These numerical oscillations are induced by the solving of the two subproblems to obtain the interface condition. However with a small tolerance these oscillations are not so sufficient to create a major difference between the two schemes. 
 %
\end{remark}
\\

The two methods present also some differences and difficulties induced by the two finite element discretizations. 
The following results show the influence of the discretization parameters on the different algorithms. 
\subsubsection{Influence of the discretization parameter $\tau$} 
The relative variation which is the maximal variation of the  integral concentration across the time steps, divided by the minimal value measures in an inverse way the global conservativity. Ideally it should be $0$ since there is no input nor output of concentration and the exact solution inside the unit disk is simply a rotation of the initial condition.
 \begin{table}[H]
\begin{center}
       \begin{tabular}{|l|l|c|r|c|c|c|c|c|c|c|c|clclclclclclclclclcl}
         \hline  
    $\tau$   &Number of iter &(min,max)&(min,max)&relat variat&relat variat 
    \\ 
     &$\qquad\leq$&STILS-DDM&STILS&STILS-DDM&STILS
   \\
       \hline
$2\pi/50$&$\qquad\;2 $&$(-0.33\;,\;0.91)$&$(-0.35\;,\;0.92)$&$3.14\cdot10^{-2}$&$2.65\cdot10^{-2}$
 \\
       $2\pi/100$&$\qquad\;3$&$(-0.39\;,\;1.06)$&$(-0.41\;,\;1.07)$&$1.60\cdot10^{-2}$&$3.98\cdot10^{-3}$
    \\ 
      $2\pi/200$	& $\qquad\;3$& $(-0.23\;,\;1.25)$&$(-0.28\;,\;1.25)$&$4.98\cdot10^{-3}$&$8.70\cdot10^{-4}$
     \\ 
      $2\pi/400$&$\qquad\;4$&$(-0.28\;,\;1.28)$&$(-0.19\;,\;1.28)$&$1.97\cdot10^{-3}$&$3.59\cdot10^{-4}$
       \\ 
      $2\pi/800$&$\qquad\;4$&$(-0.23\;,\;1.21)$&$(-0.19\;,\;1.20)$&$7.44\cdot10^{-4}$&$1.44\cdot10^{-4}$
 \\
         \hline
\end{tabular}
 \end{center} 
 \label{tab1} 
 \caption{Numerical results respecting on $\tau$}
 \end{table}
 \begin{figure}[H]
\begin{center}
\includegraphics[width=8.cm]{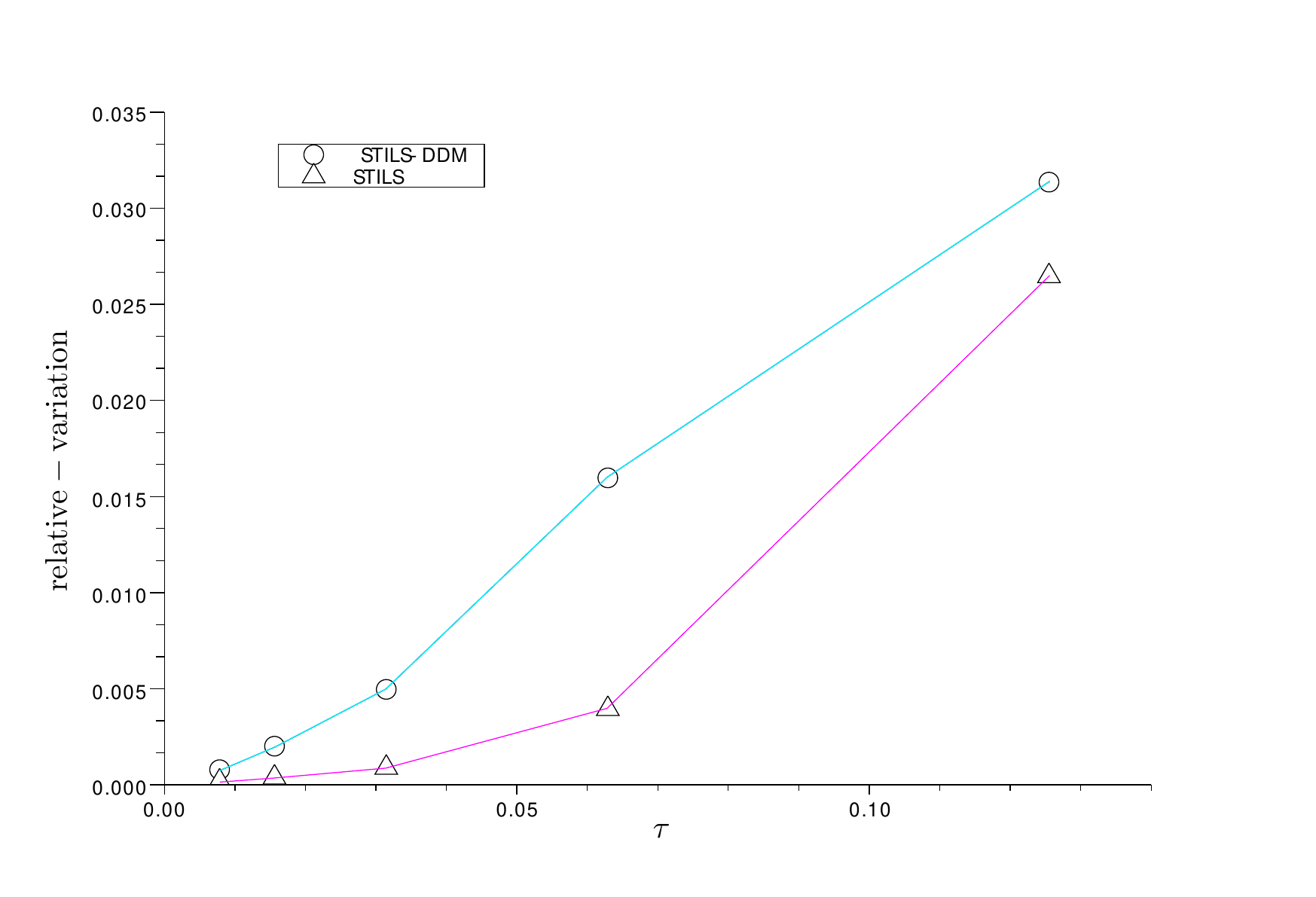}
\caption{Conservativity respecting on $\tau$}
\label{hans7}
\end{center}
\end{figure}
\begin{figure}[H]
\begin{center}
\includegraphics[width=14.cm]{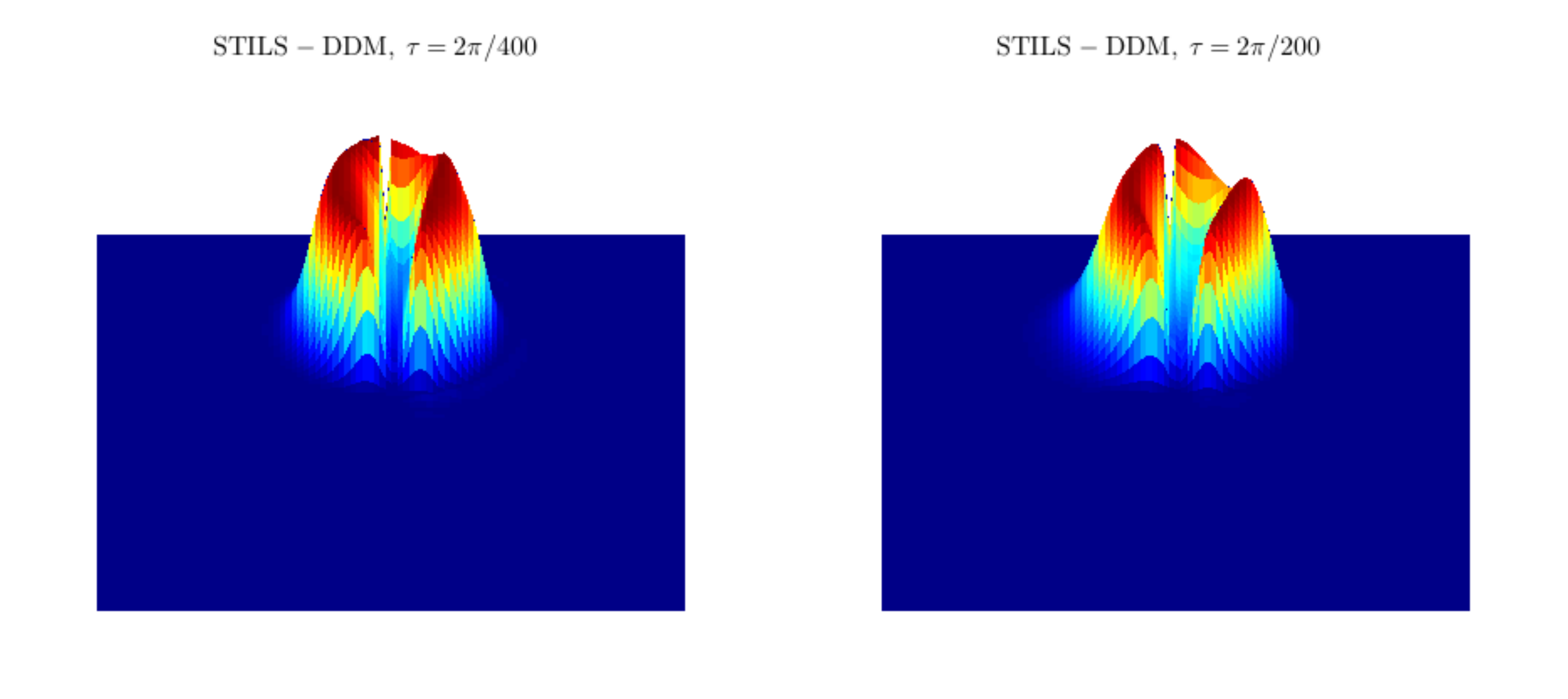}
\end{center}
\end{figure}
\begin{figure}[H]
\begin{center}
\includegraphics[width=14.cm]{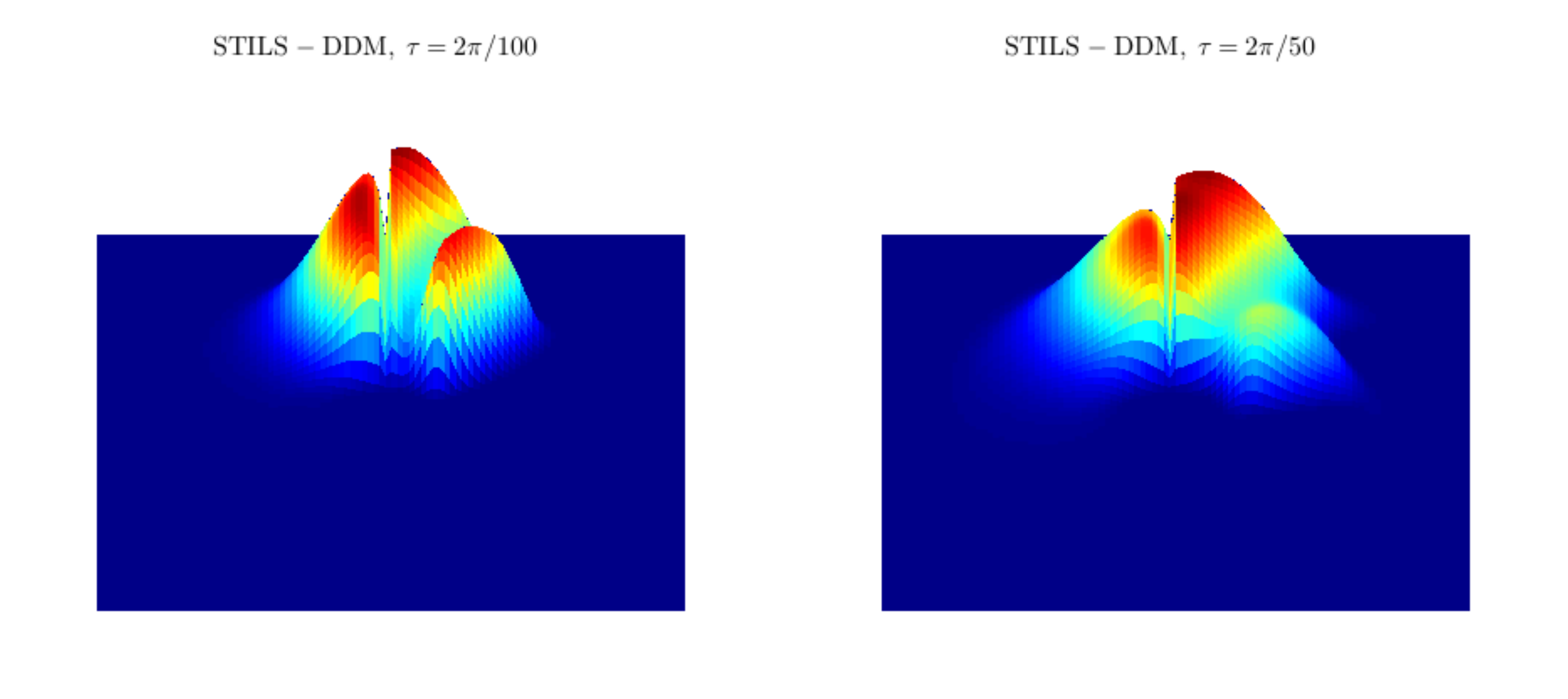}
\caption{STILS-DDM-TMQ1 numerical solutions at final time for different values of $\tau.$}
\label{hans8}
\end{center}
\end{figure}
 As showed on the table $7.1$ and fig. \ref{hans7} considering a fixed value of $h=2/100$ the global conservativity for each scheme decreases (relative variation increases) when $\tau$ increases, with a light advantage of STILS. In the STILS-DDM case this is illustrated by a deterioration of the shape of the slot at final time (see fig. \ref{hans8}).
 \subsubsection{Influence of the discretization parameter $h$} 
The same remarks are observed when $h$ takes enough high values and $\tau=2\pi/800,$ then the global conservativity decreases (see table $7.2$ and $7.3$, fig. \ref{hans9}.)
 \begin{table}[H]
\begin{center}
     \begin{tabular}{|l|l|c|r|c|c|c|c|c|c|c|c|clclclclclclclclclcl}
        \hline  
   Number of   &Number of&Numb of iterations&(min,max)&relative variation\\
   elements  &nodes&$\le$&STILS-DDM&STILS-DDM
  \\
       \hline
$1600$&$1681$&$3$&$(-0.25,1.20)$&$1.2\;10^{-2}$
 \\
       $2500$&$2601$&$3$&$(-0.20,1.30)$&$0.89\;10^{-3}$
     \\ 
      $3600$&3721& $4$&$(-0.18,1.26)$&$1.17\;10^{-3}$
    \\ 
     $10000$&$10201$&$4$&$(-0.23,1.21)$&$7.44\;10^{-4}$
       \\ 
      $40000$&$40401$&$5$&$(-0.19,1.23)$&$4.06\;10^{-5}$
      \\ 
        \hline
\end{tabular}
 \end{center}    
 \label{tab2} 
  \caption{Numerical results of STILS-DDM-TMQ1 respecting on $h$}
 \end{table}
\begin{table}[H]
\begin{center}
      \begin{tabular}{|l|l|c|r|c|c|c|c|c|c|c|c|clclclclclclclclclcl}
         \hline  
   Number of   &Number of&(min,max)&relative variation\\
   elements  &nodes&STILS&STILS
   \\
       \hline
$1600$&$1681$&$(-0.33,1.21)$&$2.76\;10^{-2}$
 \\
       $2500$&$2601$&$(-0.17,1.28)$&$5.99\;10^{-3}$
     \\ 
     $3600$&3721& $(-0.21,1.23)$&$0.98\;10^{-4}$
     \\ 
      $10000$&$10201$&$(-0.19,1.20)$&$1.44\;10^{-4}$
       \\ 
      $40000$&$40401$&$(-0.15,1.20)$&$1.87\;10^{-5}$
      \\ 
         \hline
\end{tabular}
 \end{center}    
 \label{tab3} 
  \caption{Numerical results of STILS-TMQ1 respecting on $h$}
 \end{table}
\begin{figure}[H]
\begin{center}
\includegraphics[width=8.cm]{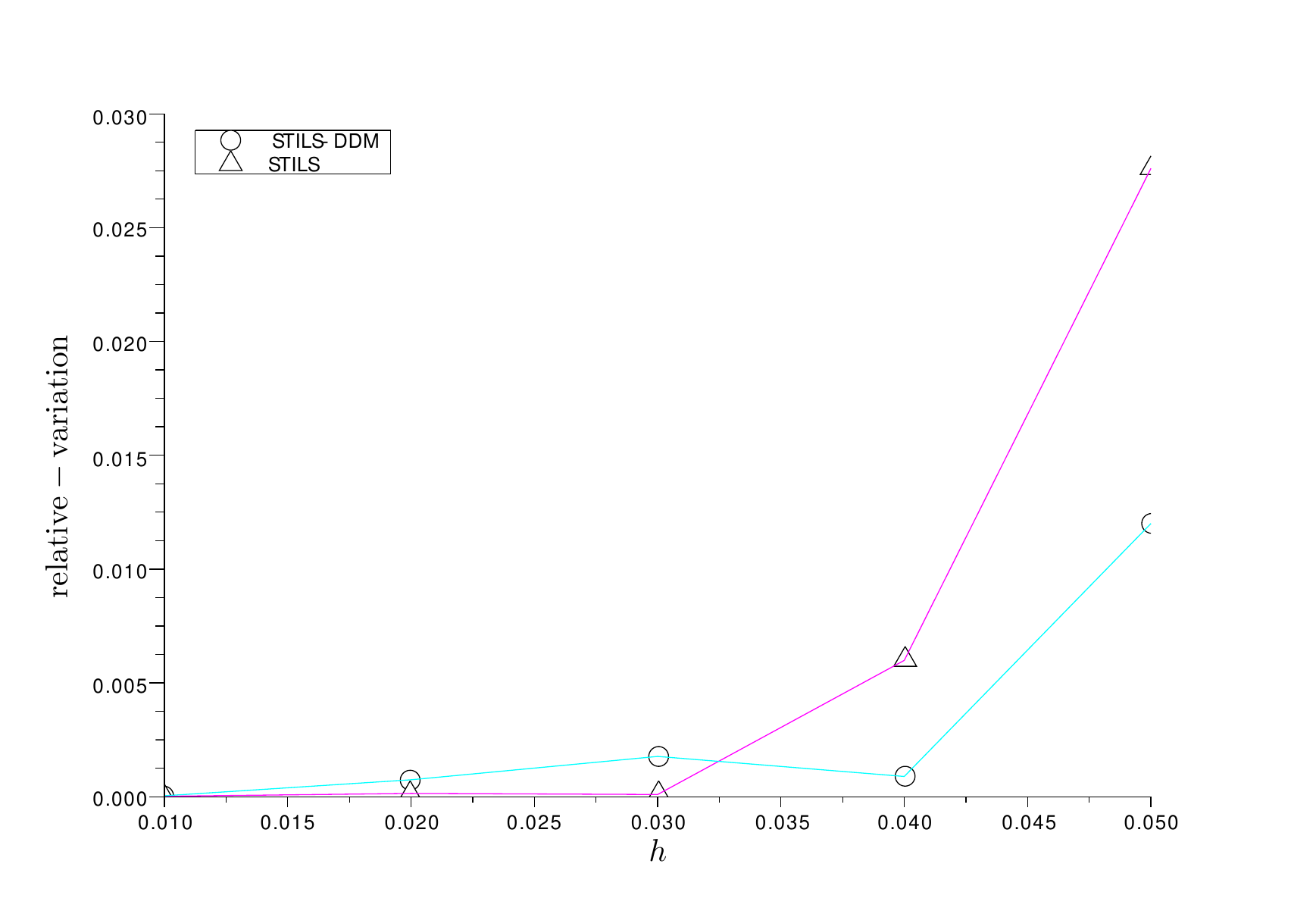} 
\caption{Conservativity respecting on $h$}
\label{hans9}
\end{center}
\end{figure}
Finally one remarks also that the two schemes induce some small  overshootings and undershootings which are nearly the same in the two cases. This due to the fact that the TM scheme is not conservative nor even TVD. At end it can be to remark that the maximal number of iterations for the STILS-DDM algorithm to obtain convergence
 varies very little with $\tau$ or $h$ (see table $7.1$, and $7.2$).\\
 
These numerical results lead to the following remark
\begin{remark}
Experimentally, in oder to product a good quality of the solution, it's necessary to have $\frac{\tau}{h}<\frac{1}{2}.$ It's look like an unstable condition while the time-marching approach of STILS is unconditionally stable see \cite{montm} .\\
When $\frac{\tau}{h}>\frac{1}{2}$ there is a phenomenon look like a numerical diffusion.
\end{remark}

\section{Conclusion}
Using the STILS method a new space-time domain decomposition algorithm to solve scalar conservation laws is presented.
A convenient way to analyze numerically this algorithm is to  discretized separately by finite element methods the space and time dimensions. The different grid parameters are very influent on the convergence of the iterative scheme. The Hansbo example illustrates well the results of this work i.e to approximate the STILS problem by a domain decomposition method. Higher orders and all sort of combinations with the techniques in Galerkin variants are possible for the proposed iteration-by-subdomain scheme. 
\section*{Acknowledgments}
I wish to express my sincere appreciation and heartfelt gratitude to Prof. O. Besson for his suggestions and the helpful discussions concerning the  subject of this work.

\end{document}